\documentclass[11pt, letterpaper]{amsart}

\usepackage{esint}
\usepackage{amssymb}

\usepackage[
hypertexnames=false, colorlinks, 
linkcolor=black,citecolor=black,urlcolor=black, linktocpage=true
]%
{hyperref} 
\hypersetup{bookmarksdepth=3}

	\normalbaselines						  %

\usepackage{graphicx}
\usepackage{
mathrsfs}

\newtheorem{thm}{Theorem}[section]
\newtheorem{lm}[thm]{Lemma}
\newtheorem{cor}[thm]{Corollary}
\newtheorem{prop}[thm]{Proposition}

\theoremstyle{definition}
\newtheorem{df}[thm]{Definition}
\newtheorem*{df*}{Definition}

\theoremstyle{remark}
\newtheorem{rem}[thm]{Remark}
\newtheorem*{rem*}{Remark}

\numberwithin{equation}{section}

\newcommand{\ci}[1]{_{ {}_{\scriptstyle #1}}}
\newcommand{\ti}[1]{_{\scriptstyle \text{\rm #1}}}

\newcommand{\cB}{\mathcal{B}}
\newcommand{\cD}{\mathscr{D}}

\newcommand{\cX}{\mathcal{X}}

\newcommand{\cN}{\mathcal{N}}
\newcommand{\cQ}{\mathcal{Q}}

\newcommand{\cG}{\mathcal{G}}
\newcommand{\cE}{\mathcal{E}}
\newcommand{\cF}{\mathcal{F}}

\newcommand{\fS}{\mathfrak{S}}
\newcommand{\f}{\varphi}

\newcommand{\e}{\varepsilon}

\newcommand{\R}{\mathbb{R}}
\newcommand{\Z}{\mathbb{Z}}
\newcommand{\N}{\mathbb{N}}
\newcommand{\E}{\mathbb{E}}

\newcommand{\bI}{\mathbf{I}}

\newcommand{\ch}{\operatorname{ch}}
\newcommand{\p}{\partial}

\newcommand{\1}{\mathbf{1}}

\newcommand{\wt}{\widetilde}

\newcommand{\cz}{Calder\'{o}n--Zygmund\ }

\newcommand{\spn}{\operatorname{span}}

\newcommand{\La}{\langle }
\newcommand{\Ra}{\rangle }

\newcommand{\sd}{{\scriptstyle\Delta}}
\newcommand{\bu}{\mathbf{u}}

\newcommand{\bff}{\mathbf{f}}
\newcommand{\bv}{\mathbf{v}}
\newcommand{\be}{\mathbf{e}}

\newcommand{\fdot}{\,\cdot\,}

\newcommand{\supp}{\operatorname{supp}}

\def\cyr{\fontencoding{OT2}\fontfamily{wncyr}\selectfont}
\DeclareTextFontCommand{\textcyr}{\cyr}

{\end{list}}      

\renewcommand{\labelenumi}{(\roman{enumi})}

\newcounter{vremennyj}

\newcommand\cond[1]{\setcounter{vremennyj}{\theenumi}\setcounter{enumi}{#1}\labelenumi\setcounter{enumi}{\thevremennyj}}

\begin{document}

\title[Entropy conditions]%
{
Entropy conditions in two weight inequalities for singular integral 
operators
}

\author[S.~Treil]{Sergei Treil}

\thanks{ST is partially supported by the NSF grant DMS-1301579}

\author[A.~Volberg]{Alexander Volberg}
\thanks{AV is partially supported by the NSF grant DMS-1265549 and by the Hausdorff Institute for Mathematics, Bonn, Germany}


\makeatletter
\@namedef{subjclassname@2010}{
  \textup{2010} Mathematics Subject Classification}
\makeatother




%
%


\begin{abstract}
The new type of ``bumping'' of the Muckenhoupt $A_2$ condition on weights is introduced. It is based on bumping the entropy integral of the weights. In particular, one gets (assuming mild regularity conditions on the corresponding Young functions) the bump conjecture proved in \cite{Le1}, \cite{NRTV1} as a corollary of entropy bumping. But our entropy bumps cannot be reduced to the bumping with Orlicz norms in the solution of bump conjecture, they are effectively smaller. Henceforth we get somewhat stronger result than the one that solves the bump conjecture in \cite{Le1}, \cite{NRTV1}. New results concerning one sided bumping conjecture are obtained.  
All the results hold in the general non-homogeneous situation. 
\end{abstract}

\maketitle

\section{Introduction}
\label{intro}

The original 
question about two weight estimates for the singular integral operators is to find a necessary and sufficient condition on the weights (non-negative locally integrable 
functions) $w$ and $v$ such that a \cz operators $T: L^p(w)\to L^p (v)$ is bounded, i.e.~the inequality
\begin{align}
\label{two-weight_01}
\int |Tf|^p v dx \le C \int |f|^p w dx \qquad \forall f \in L^p(w)
\end{align}
holds.

In the one weight case $w=v$ the famous Muckenhoupt condition is necessary and sufficient for \eqref{two-weight_01}
\begin{align}
\tag{$A_p$}
\sup_I \left(|I|^{-1}\int_{I} w dx \right) \left(|I|^{-1}\int_{I} w^{-p'/p} dx \right)^{p/p'} <\infty
\end{align}
where the supremum is taken over all cubes $I$, and $1/p+1/p'=1$. More precisely, this condition is sufficient for all \cz operators, and is also necessary for classical (interesting) \cz operators, such as Hilbert transform, Riesz transform (vector-valued, when all Riesz transforms are considered together), Beurling--Ahlfors operator.

The inequality \eqref{two-weight_01} is equivalent to the boundedness of the  operator $M_{v^{1/p}} T M_{w^{-1/p}}$ in the non-weighted $L^p$; here $M_\f$ is the multiplication operator, $M_\f f = \f f$. Denoting $u=w^{-p'/p}$ we can rewrite the problem in the symmetric form as the $L^p$ boundedness of $M_{v^{1/p} }T M_{u^{1/p'}}$.

So the problem can be stated as: \emph{Describe all weights (i.e.~non-negative  functions) $u$, $v$ such that the operator $M_{v^{1/p} }T M_{u^{1/p'}}$ is bounded in (the non-weighted) $L^p$.}

The boundedness of $M_{v^{1/p} }T M_{u^{1/p'}}$ means that 
\[
\int \left| T(u^{1/p'} g )\right|^p v dx \le \int |g|^p dx \qquad \forall g\in L^p, 
\]
which after denoting $g=u^{1/p} f$ can be rewritten as 
\begin{align}
\label{two-weight_02}
\int \left| T(u f )\right|^p v dx \le C\int |f|^p u dx  .
\end{align}
Such symmetric formulation is well known since 80s, and  for the two weight setting it looks more natural: in particular, if $T$ is an integral operator, then the integration in the operator is performed with respect to the same measure $udx$ as in the domain. This simplifies the problem, because it eliminates the third measure (the Lebesgue measure) from the considerations. 

Note also, that this  formulation is  formally more general than \eqref{two-weight_01}, because in \eqref{two-weight_01} it is usually assumed that  $w$($=u^{-p/p'}$) is locally integrable, while in \eqref{two-weight_02} we only assume local integrability of $u$; in particular, $u$ can be zero on a set of positive measure. 

It is usually assumed that in \eqref{two-weight_01} $v$ and $w$ are locally integrable, 
but  for \eqref{two-weight_02}  (respectively \eqref{two-weight_01}) to hold for interesting operators (Hilbert Transform, vector Riesz Transform, Beurling--Ahlfors Transform, etc.) the functions $v$ and $u$ (respectively $v$ and $w^{-p'/p}$) also has to be locally integrable, see \eqref{2weightA_p}, \eqref{2weightAp-01} below.

Note also, that in \eqref{two-weight_02} one can instead of weights $u$, $v$ consider Radon measures without common atoms: see \cite{Liaw-Tr_Reg_2010} for the interpretation of the singular integral operators in this case. 

For the interesting operators the following two weight analogue of the $A_p$ condition is necessary for the estimates \eqref{two-weight_01} and \eqref{two-weight_02} respectively:
\begin{align}
\label{2weightAp-01}
\sup_I \left(|I|^{-1}\int_{I} v dx \right) &\left(|I|\int_{I} w^{-p'/p} dx \right)^{p/p'}  <\infty
\\
\label{2weightA_p}
\sup_I \left(|I|^{-1}\int_{I} v dx \right) &\left(|I|^{-1}\int_{I} u dx \right)^{p/p'} <\infty .
\end{align}

Simple counterexamples show that this condition is not sufficient for the boundedness. So a natural way to get a sufficient condition is to ``bump'' the  norms, i.e.~to replace the $L^1$ norms of $u$ and $v$ in \eqref{2weightA_p}  
by some stronger  norms. 

The first idea that comes to mind is to replace $L^1$ norms by $L^{1+\e}$ norms. Namely, it was proved by C.J.~Neugebauer \cite{Neug_Insert_Ap_1983} that  \eqref{2weightAp-01} with $v^{1+\e}$, $w^{1+\e} $ instead of $v$ and $w$ is sufficient for the boundedness of \cz operators, because in this case one can insert an $A_p$ weight between $v$ and $Cw$.

More generally, given a Young function $\Phi$ and a cube $I$ one can consider the normalized on $I$ Orlicz space $L^\Phi(I)$ with the norm given by
\begin{align*}
\|f\|\ci{L^\Phi(I)} := \inf\left\{ \lambda>0 : \int_I \Phi\left( \frac{ f(x)}{\lambda } \right) \frac{dx}{|I|} \le 1 \right\}.
\end{align*}
And it was conjectured (for $p=2$) that  if the Young functions $\Phi_1$ and $\Phi_2$ are integrable near infinity,
\begin{equation}
\label{inte}
\int^\infty \frac{dx}{\Phi_{i}(x)} <\infty,\,\, i=1,2\,,
\end{equation}
then the condition
\begin{align}
\label{bump_01}
\sup_I \| v\|\ci{L^{\Phi_1}(I)} \| u\|\ci{L^{\Phi_2}(I)} < \infty
\end{align}
implies that  for any bounded \cz operator $T$
the operator $M_{v^{1/2}} T M_{u^{1/2}}$ is bounded in $L^2$. Usually in the literature a more complicated (although equivalent) form of this conjecture was presented, but at least in the case $p=2$ condition \eqref{bump_01} seems more transparent.%
\footnote{The bump condition was also  stated for $p\ne 2$, but  in this paper we only deal with the case $p=2$.}

Condition \eqref{bump_01} was considered in numerous papers in the attempt to prove its universal sufficiency for \emph{all} \cz operators. The reader can find beautiful approaches in  \cite{CU-Ma-Pe}, \cite{CU-Pe99},  \cite{CU-Pe00},   \cite{Le},  \cite{Pe94JL}, \cite{PeMax},  where partial results for some \cz operators were proved (note that \cite{PeMax} is about maximal operator and not about \cz operators). Finally in \cite{Le1} the sufficiency of bump condition for all \cz operators to be bounded was fully proved (and even generalized to all $p\in (1, \infty)$), although in formally less general situation of the estimates  \eqref{two-weight_01}. Simultaneously and by different  methods of Bellman function this bump conjecture was proved in \cite{NRTV1}.

\begin{rem*}
We should mention here a breakthrough paper \cite{La} where it was proved that the estimate \eqref{two-weight_02} holds for $p=2$ if and only if \eqref{two-weight_02} and the ``dual'' estimate with $u$ and $v$ interchanged hold uniformly for indicators of intervals (and the so-called Poisson $A_2$ condition is satisfied); this result solves a long standing conjecture by Nazarov--Treil--Voberg. 

Paper \cite{La} was a culmination of the long line of research started by Nazarov and us \cite{NTVlost} (see also the last two chapters of a book \cite{Vo}), and continued in joint works of Lacey, Sawyer, Uriarte-Tuero, and Shen: \cite{LSUT}, \cite{LSUT1}, \cite{LSSUT}. In particular, a very interesting achievment of the latter group was the introduction of a certain ``energy", a very interesting positive quantity capturing ``the specifics of the tail" of the Hilbert transform. The energy condition did this better than the similar condition in \cite{NTVlost} (see also \cite{Vo}). In particular, the energy condition turned out to be a necessary one for the two-weight boundedness of the Hilbert transform. 

\medskip

But in this paper we are looking for a condition that is \emph{universal}, i.e.~sufficient for the estimates for the class of \emph{all} \cz operators (or for its reasonable subclass). We would also  want this condition to be reasonably simple, for example, it should involve only the testing of weights, and not the testing  of  operators themselves: we think that we achieved  practically ultimate results  of this kind in the present article.

It would also be extremely interesting to find a sufficient condition that is also necessary, meaning that it follows from \eqref{two-weight_02} for \emph{all} \cz operators (or some reasonable subclass).  However, we do not treat necessity in this paper (however, we treat the sharpness of our results), and we are afraid this might be a very hard problem. 
\end{rem*}

In this paper we give a quite different way of ``bumping'' the weights. We use what we call  the ``entropy bumps". The resulting bumps are effectively smaller than the ones used in \eqref{bump_01} (see Lemmas \ref{l:Orl-Entr}, \ref{l:volb}).

To explain what is the entropy bump condition, let us introduce
for a weight $u$  
\[
\bu\ci I:= \La u\Ra\ci I = \|u\|\ci{L^1(I)},
\qquad \bu^*\ci I:=  \|u\|_{L\log L(I)} \approx \| M \1\ci I u \|\ci{L^1(I)} ,  
\] 
and similarly for a weight $v$. 

Function $\Phi_0(t) = t\log^+ t$ does not satisfy \eqref{inte} of course, 
and the condition on two weights
\[
\sup_I\bv^*\ci I \bu^*\ci I<\infty
\]
is  not sufficient condition for the estimate \eqref{two-weight_02}, see Section \ref{s:Sharpness} below. . 

To make up for the ``smallness" of $L\log L$-bumps $\sup_I\bv^*\ci I\bu^*\ci I$ we introduce the ``penalty term''.

Namely, let  $\alpha:[1, \infty)\to \R_+$ be a function such that $t\mapsto t\alpha(t)$ is increasing and 
\begin{align*}
C_\alpha:=\int_1^\infty \frac1{t\alpha(t)} dt <\infty. 
\end{align*}
The entropy bump of the weight $u$ is the following quantity
\begin{align}
\label{entbump}
\cE^\alpha_I(u):=\bu\ci I^*\alpha\left(\bu^*\ci I/\bu\ci I\right) ,.
\end{align}

We will prove that the entropy bump condition
\begin{equation}
\label{entbumpCond}
\sup_I \cE^\alpha_I(u) \cE^\alpha_I(v)<\infty
\end{equation}
is already sufficient for the boundedness of $M_{v^{1/2}} T M_{u^{1/2}}$ for any Calder\'on--Zygmund operator, as soon as $C_\alpha<\infty$.

Moreover, given two Orlicz functions $\Phi_1, \Phi_2$ (with a mild extra regularity) satisfying \eqref{inte} and two weights $u,w$ satisfying \eqref{bump_01} one can show the existence of $\alpha$, with $C_\alpha<\infty$ such that the entropy bump condition \eqref{entbumpCond} is satisfied, see Lemma \ref{l:Orl-Entr} below.

This shows that the results of the present paper give those of \cite{NRTV1}, \cite{Le1} (for $p=2$).

But not the other way around because the weights $u, v$ are allowed to be only in $L\log L$ locally by the condition \eqref{entbumpCond}, while in \eqref{bump_01} the weights are required to be in much smaller Orlicz spaces.

We will also prove a partial result concerning the one sided bump conjecture. Namely, we will show that if 
\begin{align}
\label{1sided-01}
\sup_I \alpha\left(\bu^*\ci I/\bu\ci I\right)^2\bu^*\ci I \bv\ci I <\infty, \qquad \sup_I \alpha\left(\bv^*\ci I/\bv\ci I\right)^2\bv^*\ci I \bu\ci I <\infty, 
\end{align}
then for the sparse (Lerner type) operators 
the weighted estimate \eqref{two-weight_02} is satisfied uniformly. This is a partial result in the direction of the so-called one sided bump conjecture. 
Again, in  light of Lemma \ref{l:Orl-Entr} below, this result is stronger than the corresponding result in \cite{NRV_1Sbumps} or the result for $p=2$ in \cite{La_bumps}
(for Young functions with mild regularity). 

If one could eliminate the exponent $2$ at $\alpha$, then it would give the one-sided bump conjecture for sufficiently  regular Young functions. However, our proofs give us the exponent $2$, and at the moment we do not know if it is an artifact of the proof, or if it is essential. 
At the moment we are inclined to think that this is essential.

\begin{rem*}
All the result can be stated with two penalty functions $\alpha_{1,2}$. However, defining $\alpha(s)= \min\{\alpha_1(s), \alpha_2(s)\}$, one immediately gets the results for two penalty functions $\alpha_{1,2}$ from the results with one $\alpha$. 
\end{rem*}

\section{General setup and main results}

\subsection{Setup and main definitions. }
\label{s:setup}
Consider a $\sigma$-finite measure space $(\cX, \fS, \mu)$ with the filtration (i.e.~with the sequence of increasing $\sigma$-algebras) $\fS_n$, $n\in\Z$, $\fS_n\subset\fS_{n+1}$. 

We assume that each $\sigma$-algebra $\fS_n$ is \emph{atomic}, meaning that 
there exists a countable \emph{disjoint} collection $\cD_n$ of the sets of positive measure (atoms), such that every $A\in\fS_n$ is a union of sets $I\in\cD_n$. 

The fact that $ \fS_n\subset \fS_{n+1}$ means that every $I\in \cD_n$ is at most countable union of $I'\in\cD_{n+1}$. 

We denote by $\cD=\bigcup_{n\in\Z}\cD_n$ the collection of all atoms  (in all generations). 

The typical example will be the filtration given by a dyadic lattice in $\R^d$, so the notation $\cD$. Note, that we do not assume any homogeneity  in our setup, so the more interesting example will be the same dyadic lattice in $\R^d$, but the underlying measure is an arbitrary Radon measure $\mu$.  

We will allow a situation  when an atom $I$ belong to several (even infinitely many) generations $\cD_n$, so the case of dyadic lattice in a cube is also covered. 
However, we will not allow $I$ to be in all generations, because in this case nothing interesting happens on the interval $I$. 

We  usually will not assign a special symbol for the underlying measure, and use $|A|$ instead of $\mu(A)$, and $dx$ instead of $d\mu(x)$ in integrals. 

Note that our filtered space $\cX$ can be represented as a countable (finite or infinite) direct sum of the filtered spaces treated in \cite{Tr_Comm-para2010}, so all the results from \cite{Tr_Comm-para2010} hold in our case. 

\begin{df}
\label{df:child} Let $I\in\cD$, and let $n\in\Z$ be a maximal integer such that $I\in \cD_n$. Then $I'\in\cD_{n+1}$ such that $I'\subset I$ are called the \emph{children} of $I$. The collection of all children of $I$ will be denoted by $\ch(I)$. 

If $I\in \cD_n$ for all sufficiently large $n$ we set $\ch(I)=\{I\}$. 

The (grand)children $\ch_n(I)$ of order $n$ can be defined inductively, $\ch_1(I) :=\ch(I)$, 
\[
\ch_n(I) := \bigcup_{I'\in\ch_{n-1}(I)} \ch(I'). 
\]
We also formally define $\ch_0(I):=\{I\}$. 
\end{df}

\subsubsection{Martingale differences, Haar shifts and paraproducts}
For a measurable $I$ we define the average
\[
\La f\Ra\ci I :=|I|^{-1} \int_I f dx, 
\]
and the averaging operator $\E\ci I$ by
\[
\E\ci I f = \La f\Ra\ci I \1\ci I. 
\]
For $I\in\cD$ the \emph{martingale difference operator} $\Delta\ci I$ is given by
\[
\Delta\ci I f := \sum_{I'\in\ch(I) } \E\ci{I'} f \  - \ \E\ci I f ; 
\]
note that formally 
if $\ch(I)=\{I\}$ then 
 $\Delta\ci I =0$. 
 
The $n$th order martingale difference $\Delta^n\ci I$ is given by 
\[
\Delta^n\ci I := \sum_{I'\in\ch_n(I)} \E\ci{I'} \ - \E\ci I = \sum_{\substack{I'\in\ch_k(I) \\ 0\le k <n } } \Delta\ci{I'}\,.
\]

Denote $D\ci I:= \Delta\ci I L^2$ (note that $D\ci I=\{0\}$ if $\ch(I)=\{I\}$), and $D^n\ci I := \Delta\ci I^n L^2$.  

\begin{df}
An operator $T$ on $L^2=L^2(\cX)$ is called a \emph{Haar shift} of order (complexity) $n$ if
\[
Tf = \sum_{I\in\cD} T\ci I (\Delta\ci I^n f)
\]
where $T\ci I$ is an operator on $D^n\ci I$ such that 
\begin{align}
\label{L1xL1-norm}
(T\ci I f, g)\ci{L^2} \le |I|^{-1} \|f\|_1\|g\|_1 \qquad \forall f, g\in D\ci I. 
\end{align}
\end{df}

This definition is formally a bit more general than the one used in \cite{TB2011}, where it was assumed that $T\ci I$ can be represented as an integral operator with kernel $a\ci I$, $\supp a\ci I \subset I\times I$, $\|a\ci I\|_\infty \le |I|^{-1}$. 

Also, a Haar shift of complexity $1$ is what is often called a \emph{martingale transform}; this case is special because  the blocks $T\ci I$ are acting on mutually orthogonal subspaces. Note, that what was called the martingale transform in \cite{Tr_Comm-para2010}  is almost exactly the Haar shift of complexity  $1$, with the only difference that here we use a stricter normalization condition on the blocks $T\ci I$. 

The importance of Haar shifts comes from the fact that in the classical case of dyadic lattices in $\R^d$, any \cz operator can be represented as a weighted average (over all translations of the standard dyadic lattice in $\R^d$) of the paraproducts and the Haar shift, with the weights decreasing exponentially in the complexity of the shifts, see for example \cite{H,HPTV}. This means that the uniform boundedness of the paraproducts, see Definition \ref{df:para} below,  together with bounds on the Haar shifts that grow sub-exponentially in the complexity of the shifts imply the boundedness of the \cz operators. 

In this paper we only consider the Haar shifts of order $1$, because by modifying the filtration,  every Haar shift of complexity $n$ can be represented as a sum of $n$ Haar shifts of complexity $n$ (each with respect to its own filtration). 

Namely, if we consider filtrations $\cF^r$, $r=0, 1, \ldots, n-1$ defined by generations $\cD_{r+k}$, $k\in\Z$, then a Haar shift $T$ can be represented as the sum $T=\sum_{r=1}^n T_r$, where $T_r$ is a Haar shif of complexity $1$ (martingale transform) with respect to the filtration $\cF^r$. 
This splitting is trivial if each $I\in\cD$ has a non-trivial collection of children: in this case each interval $I$ belongs to a unique generation $\cD_j$, and thus the block $T\ci I$ can be canonically  assigned to a unique filtration. The general case is just a bit more complicated: since an interval $I$ can be in more than $1$ generation, the corresponding block $T\ci I$ can be assigned to more than $1$ filtration $\cF^r$; we just need to assign it only to one of the possible choices.   

So, an estimate for the Haar shifts of complexity $1$ give the estimate for general Haar shifts that grow linearly in complexity, which is more than enough for the estimates of \cz operators.

\begin{df}
\label{df:para}
Let $b= (b\ci I)\ci{I\in\cD}$, $b\ci I \in D\ci I$. A \emph{paraproduct} $\Pi=\Pi_b$ is an operator  on $L^2$  given by 
\[
T f := \sum_{I\in\cD} \La f\Ra\ci I b\ci I . 
\]
\end{df}

Often $b\ci I = \Delta\ci I b$ for some function $b$, but we will not distinguish between the cases when $b$ is a sequence and when it is a function.

It follows immediately from the martingale Carleson embedding theorem that the paraproduct $\Pi_b$ is bounded in $L^2$ if and only if the sequence $\{\|b\ci I\|_2^2\}_{I\in\cD}$ satisfies the \emph{Carleson measure condition}
\begin{align}
\label{CMC-para}
\left\| b \right\|\ti{Carl}:=
\sup_{I_0\in\cD} |I_0|^{-1} \sum_{I\in\cD, \, I\subset I_0} \|b\ci I\|_2^2  <\infty
\end{align}
In what follows we always normalize the paraproducts by assuming $\|b\|\ti{Carl} \le 1$. 

\subsubsection{Sparse operators}

A family $\cQ\subset \cD$ is called \emph{sparse} if for any $I_0\in\cQ$
\[
\biggl| \bigcup_{I\in\cQ, \, I\subsetneqq I_0} I \biggr| \le \frac12 |I_0|. 
\]

\begin{rem*}
The constant 1/2 is really not essential here, in the above inequality one could use any $c<1$, because a sparse collection with a bigger constant can be represented by a finitely many sparse collections with a smaller constant. This can be achieved by skipping levels, so the number of collection depends only on ratio of the constants. 
\end{rem*}

\begin{df}
Let $\cQ$ be a sparse collection. The sparse, or Lerner type operator $T\ci\cQ$ is defined by 
\[
T\ci\cQ f := \sum_{I\in\cQ}\La f\Ra\ci I \1\ci I \,.
\]
\end{df}

The importance of sparse operators comes from the result of A.~Lerner \cite{Le1} which states that for any Banach function space $X$ over $\R^d$ and for any \cz operator $T$
\[
\| T_\sharp f\|\ci X \le C(T, d) \sup_{\cD, \cQ\subset \cD} \|T\ci{\cQ} f\|\ci X
\]
where supremum is taken over all dyadic lattices $\cD$ in $\R^d$ and over all sparse collections $\cQ\subset\cD$; here $T_\sharp$ is the maximal \cz operator, 
\[
T_\sharp f(x) = \sup_{0<\e<R<\infty} \left| \int_{y:\e<|x-y|<R} K(x,y) f (y) dy \right|\,.
\]

Thus a weighted estimate for  sparse operators in $\R^d$ implies the corresponding estimate for general \cz operators.  

A corresponding result is known to be true for operators on homogeneous spaces, see \cite{Ander-Vagar_A_2_2012}. 

Note, that the above result was obtained by estimating the Haar shifts and paraproducts, so in the homogeneous situation the estimate for the sparse operators imply the estimate for Haar shifts and paraproducts (and, thus, for arbitrary Calder\'on--Zygmund operators). 

However, in the non-homogeneous case the reduction to the sparse operators is not known, so we present proofs for  Haar shifts and paraproducts. 

The proof we present for the Haar shifts  utilizes the Bellman function methods and is of independent interest. 

\subsection{Main results}

\begin{thm}
\label{t:main-01}
Let $T$ be either  a Haar shift of complexity $1$, or a sparse operator, or a paraproduct, normalized by the condition 
\begin{align}
\label{CMCpara-01}
\sup_{I\in\cD} |I|^{-1} \sum_{I'\in\cD\,,I'\subset I} \| b\ci{I'}\|_\infty^2 |I'| \le 1,  
\end{align}
and let weights $u$, $v$ satisfy the condition 
\begin{align}
\label{entr-bump-01}
\sup_{I\in\cD} \alpha\left(\bu^*\ci I/\bu\ci I\right) \bu\ci I^* \alpha\left(\bv^*\ci I/\bv\ci I\right) \bv\ci I^* :=A <\infty, 
\end{align}
where $\alpha:[1, \infty)\to \R_+$ is such that the function $t\mapsto t\alpha(t)$ is increasing and  
\begin{align}
\label{C_alpha}
C_\alpha:= \frac{1}{\alpha(1)}+\int_1^\infty \frac1{t\alpha(t)} dt <\infty. 
\end{align}
Then the operator $M_v^{1/2} T M_u^{1/2}$ is bounded in $L^2$, or equivalently
\begin{align}
\label{inte-02}
\int_\cX |T(fu)|^2 v dx \le C\int_\cX |f|^2 u dx \qquad \forall f\in L^2(u). 
\end{align}
\end{thm}

\begin{rem*}
Note, that in the homogeneous case, when 
\[
|I'|/|I| \ge\delta >0 \qquad \forall I\in \cD, \ \forall I'\in\ch(I), 
\]
$\|\Delta \ci I f\|^2_\infty |I| \asymp \|\Delta\ci I f\|_2^2$ uniformly for all $f\in L^2$ and for all $I\in\cD$ (this in fact can be used as a definition of homogeneous lattices), so in the homogeneous case the paraproduct   normalization condition \eqref{CMCpara-01} is equivalent up to a constant to the classical normalization condition \eqref{CMC-para}. 

Condition \eqref{CMC-para}, as we mentioned before, is equivalent to the bound on the norm of the paraproduct in the non-weighter $L^2$; our condition \eqref{CMCpara-01} is a bit stronger in the general non-homogeneous case. 
\end{rem*}

\begin{thm}
Let $T$ be a sparse (a Lerner type) operator, and let the weights $u$, $v$ satisfy the separated bump conditions
\begin{align}
\label{1sided-02}
\sup_{I\in\cD} \alpha\left(\bu^*\ci I/\bu\ci I\right)^2\bu^*\ci I \bv\ci I <\infty, \qquad \sup_{I\in\cD} \alpha\left(\bv^*\ci I/\bv\ci I\right)^2\bv^*\ci I \bu\ci I <\infty,  
\end{align}
where again $\alpha:[1, \infty)\to \R_+$ is an increasing function satisfying \eqref{inte-02}.

Then the operator $M_v^{1/2} T M_u^{1/2}$ is bounded in $L^2$, i.e.
\begin{align*}
\int_\cX |T(fu)|^2 v dx \le C\int_\cX |f|^2 u dx \qquad \forall f\in L^2(u). 
\end{align*}
\end{thm}

\section{Lorentz spaces and class  \texorpdfstring{$L \log L$}{L log L}}
\label{s:H1-norm}

Let us recall that the Lorentz space $\Lambda_\psi= \Lambda_\psi(\cX,\mu)$ is defined as the set of measurable functions $f$ on $\cX$ such that 
\begin{align}
\label{La_psi}
\| f\|\ci{\Lambda_\psi} := \int_0^{\mu(\cX)} f^*(s) d\psi(s) < \infty;
\end{align}
here $f^*$ is the non-decreasing rearrangement of the function $|f|$, and $\psi$ is a quasiconcave function. 

Recall, see \cite[Ch.~2, Definition 5.6]{Bennett-Sharpley_1988} that a function $\psi:[0,\infty) \to [0,\infty)$ \emph{quasiconcave} if

\begin{enumerate}
\item $\psi$ is increasing;
\item $\psi(s) = 0$ iff $s=0$;
\item $s\mapsto \psi(s)/s$ is decreasing. 
\end{enumerate}

Note, that conditions \cond1 and \cond3 imply that the function $\psi$ is continuous for all $s>0$; it can have a jump discontinuity at $s=0$, but in this paper we will will only consider continuous functions $\psi$. Note also, that an increasing concave $\psi$ satisfying  $\psi(0)=0$ is also quasiconcave. 

The following simple proposition, see \cite[Ch.~2, Proposition 5.10]{Bennett-Sharpley_1988}, shows that without loss of generality we can assume that the function $\psi $ is concave.

\begin{prop}
\label{p:LCM}
If $\psi$ is quasiconcave, then the least concave majorant $\wt\psi$ of $\psi$ satisfies 
\begin{align}
\label{LCM}
\frac12 \wt\psi(s) \le \psi(s) \le \wt\psi(s) \qquad \forall s\ge 0. 
\end{align}
\end{prop}

If the function $\psi $ is concave, the expression \eqref{La_psi} indeed defines a norm on the space of functions $f$ satisfying \eqref{La_psi}.

In this paper $\mu$ will always be a probability measure without atoms, particularly the normalized Lebesgue measure on an interval.

If $N=N_f$ is the distribution function of $f$, 
\[
N_f(t) := \mu\{ x\in\cX: |f(x)|>t\}, 
\]
then making the change of variables $s=N(t)$ and integrating by parts we can rewrite \eqref{La_psi} as 
\begin{align}
\label{La_psi-01}
\| f\|\ci{\Lambda_\psi} = \int_0^\infty \psi(N(t)) dt = \int_0^\infty N(t) \Psi(N(t)) dt ;
\end{align}
here we define $\Psi(s):=\psi(s)/s$. This calculation is definitely justified for continuous $f^*$, so $N$ is the inverse of $f^*$; approximating a general $f^*$ by an increasing sequence of continuous decreasing functions gives \eqref{La_psi-01} for all $f$.

\subsection{Comparison of Lorentz spaces with other rearrangement invariant spaces}
This section, except for Lemma \ref{l:Orl-Lor} below, is not necessary for the proofs. We present here a well-known facts helping to put the results in perspective, so it could be still beneficial for the reader.

Recall that for a rearrangement invariant Banach function space $X$, its \emph{fundamental function} $\psi=\psi\ci X$ is defined as 
\begin{align}
\label{fund-funct}
\psi\ci X(s) = \| \1\ci E\|\ci X, \qquad \mu(E)=s. 
\end{align}

Note, that the fundamental function for the Lorentz space $\Lambda_\psi$ is exactly $\psi$. 

It is an easy calculation (see also \cite[Ch.~4, Lemma 8.17]{Bennett-Sharpley_1988}) that for the Orlicz space $L^\phi$ equipped with the Luxemburg norm its fundamental function $\psi$ is given by 
\begin{align}
\label{psi-L^p}
\psi(s) = 1/\Phi^{-1}(1/s);
\end{align}
in particular, the fundamental function for $L^p$ spaces is given by $t^{1/p}$. 

For ``bumping'' the Muckenhoupt condition the spaces $\Lambda_\psi$ are easier to work with  than the Orlicz spaces $L^\Phi$ traditionally used for this purpose. 

To be able to replace the Orlicz norm, one has to estimate the norm in $L^\Phi$ below by the norm in an appropriate  Lorentz space $\Lambda_\psi$. 

In \cite{NRTV}, \cite{NRTV1} the following comparison of the Lorentz and Orlicz  norms was obtained. 

\begin{lm}
\label{l:Orl-Lor}
Let the underlying measure space $(\cX, \mu)$ be a probability space, i.e.~$\mu(\cX)=1$. 
Let $\Phi$ be a Young function such that 
\[
\int^\infty \frac{dt}{\Phi(t)} <\infty, 
\]
and let the function $\Psi$ on $(0,1]$ be defined parametrically 
\[
\Psi(s) := \Phi'(t) \qquad \text{for}\quad s=\frac{1}{\Phi(t)\Phi'(t)}. 
\]
Then the function $\psi$, $\psi(s) = s \Psi(s)$ is quasiconcave, satisfies
\[
\int_0^1 \frac{ds}{\psi(s)} < \infty, 
\]
and there exists  $C<\infty$ such that $\|f\|\ci{\Lambda_\psi} \le C \|f\|_{L^\Phi}$ for all measurable $f$, 
\end{lm}

In the journal version of \cite{NRTV1} it is shown that for Young's functions with extra regularity the norms are equivalent.

\subsection{Space \texorpdfstring{$L\log L$}{L log L} as a Lorentz space}
\label{s:LlogL-Lorentz}

Recall that the space $L \log L$ is usually defined as the Orlicz space with the Young function $\Phi_0(t) =t\log^+(t)$ (the function $\Phi(t)= t\log(1+t)$ is also used and gives an equivalent norm). If the underlying measure space $(\cX, \mu)$ satisfies $\mu(\cX)=1$, the space $L\log L$ can also be defined as the Lorentz space $\Lambda_{\psi_0}$ with $\psi_0(s) = s\ln(e/s)$, $s\in[0,1]$,  see \cite[Ch.~4, Sections 6, 8]{Bennett-Sharpley_1988}. 

If the underlying space is a unit interval $I_0$ (with Lebesgue measure), then it is well-known that $\|Mf\|_1$, where $M$ is the Hardy--Littlewood maximal function, defines an equivalent norm on $L\log L$. Recall that the Hardy--Littlewood maximal function $M =M_{I_0}$ on and interval $I_0$ is defined by 
\[
Mf (s) = \sup_{I:s\in I} |I|^{-1} \int_I |f(x)| dx, 
\]
where the supremum is taken over all intervals $I\subset I_0$, $s\in I$.

Moreover, it is well known, see \cite[Ch.~4, eqn. (6.3)]{Bennett-Sharpley_1988} that for $\psi_0(s)=s\ln(e/s)$ and $\mu(\cX)=1$
\begin{align}
\label{LlogL-max01}
\|f\|\ci{\Lambda_{\psi_0}} = \int_0^1 Mf^*(s) ds. 
\end{align}
This implies that for any collection $\cF$  of measurable subsets of $\cX$ we have the estimate for the corresponding maximal function $M\ci\cF$, 
\[
\int_\cX M\ci\cF f \le \|f\|\ci{\Lambda_{\psi_0}};
\]
here 
\[
M\ci\cF f(x):= \sup_{I\in \cF:\, x\in I} |I|^{-1}\int_I |f|d\mu.
\]
\subsection{Comparison of Orlicz and entropy  bumps}
Recall that a Young function $\Phi$ is called \emph{doubling}  (or satisfying $\Delta_2$ condition) if there is a constant $C<\infty$ such that $\Phi(2t) \le C \Phi(t)$ for all sufficiently large $t$. We will use this definition for arbitrary increasing functions, without requiring the convexity. 

In this section we assume that the underlying measure space is a probability space, i.e.~that $\mu(\cX)=1$. 

Fix a norm $\|\fdot\|_*$  in $L\log L$, namely let $\|f\|_*:= \|f\|\ci{\Lambda_{\psi_0}}$, $\psi_0(s) = s \ln(e/s)$.  
\begin{lm}
\label{l:Orl-Entr}
Let $\Phi$ be a doubling Young function satisfying
\begin{align}
\label{conv:1/Phi}
\int^\infty \frac{dt}{\Phi(t)} <\infty
\end{align}
and such that the function $t\mapsto \Phi(t)/(t\ln t)$ is increasing for sufficiently large $t$. Then there exists an increasing function $t\mapsto t\alpha(t)>0$ on $[1,\infty)$ such that 
\begin{align*}
\int_1^\infty \frac{dt}{t\alpha(t)} &< \infty \\
\intertext{and}
\alpha\left(\frac{\|f\|_*}{\|f\|_1} \right) \|f\|_* &\le \|f\|_{L^\phi}. 
\end{align*}
\end{lm}

The assumption that the function $t\mapsto \Phi(t)/(t\ln t)$ is increasing seems not too restrictive, especially in the light of Lemma \ref{l:Pphi/LlogL} below. It is also satisfied by the standard logarithmic bumps of $t$, namely for the Young functions
\[
\Phi(t) = t \ln t \ln_2 t \ldots \ln_{n-1}t (\ln_n t)^{1+\e}, \qquad \e>0;
\]
here $\ln_{k+1}(t) =\ln(\ln_k(t))$ and $\ln_1(t):=\ln t$. 

\begin{lm}
\label{l:Pphi/LlogL}
Let a Young function $\Phi$ satisfies \eqref{conv:1/Phi}. Then there exists $c>0$ such that
\[
\Phi(t) \ge c \cdot t\ln t 
\]
for all sufficiently large $t$. 
\end{lm}

\begin{proof} By convexity we have $\Phi(t)\le t \Phi'(t)$. Therefore \eqref{conv:1/Phi} implies
$$
\int_1^t \frac{dt}{t\Phi'(t)} \le C_1<\infty
$$
independently of $t$. As $1/\Phi'(t)$ decreases, we write that the left hand side is at least $\frac{\ln t}{\Phi'(t)}$. Then $\Phi'(t) \ge c_1 \ln t$, where $c_1= C_1^{-1}$. Lemma \ref{l:Pphi/LlogL} follows.
\end{proof}

\bigskip

If $\Phi$ is a Young function, it is an easy corollary of convexity that 
\begin{align}
\label{Phi<tPhi'}
\Phi(t) \le t\Phi'(t) 
\end{align}


If $\Phi$ is doubling, it is an easy corollary of the representation 
\begin{align}
\label{FTC-Phi}
\Phi(t) = \int_0^t \Phi'(x) dx
\end{align}
that there exists $c>0$ such that
\begin{align}
\label{Phi>ctPhi'}
\Phi(t) \ge c t \Phi'(t)
\end{align}
for all sufficiently large $t$. If a Young function $\Phi$ is doubling, we can conclude from \eqref{Phi<tPhi'} and \eqref{Phi>ctPhi'} that $\Phi'$ is also doubling. On the other hand, if $\Phi'$ is doubling, then representation \eqref{FTC-Phi} implies that $\Phi$ is doubling as well. 


\begin{lm}
\label{l:min_bump}
Let $\Phi_{1,2}$ be doubling Young functions such that
\[
\int^\infty \frac{dt}{\Phi_{1,2}(t)} <\infty. 
\]
Then there exists a doubling Young function $\Phi$, $\Phi(t) \le \min\{\Phi_1(t), \Phi_2(t)\}$ such that 
\begin{align}
\label{int1/Phi-01}
\int^\infty \frac{dt}{\Phi(t)} <\infty .
\end{align}
Moreover, there exists $c>0$ such that 
\begin{align}
\label{Phi>min}
\Phi(t) \ge c \min\{\Phi_1(t), \Phi_2(t)\}
\end{align}
for all sufficiently large $t$. 
\end{lm}

\begin{proof}
Let $\phi_{1,2}:=\Phi_{1,2}'$, so 
\[
\Phi_{1,2}(t) = \int_0^t \phi_{1,2}(x) dx. 
\]
Define 
\[
\phi(t):= \min\{\phi_1(t), \phi_2(t)\}, \qquad \Phi(t) :=\int_0^t \phi(x) dx . 
\]
Clearly $\phi$ is increasing and doubling, so $\Phi$ is a doubling Young function. Since $\phi(t)\le \phi_{1,2}(t)$ we can conclude that $\Phi\le\min\{\Phi_1, \Phi_2\}$. 

Finally, since all Young functions are doubling, we get using \eqref{Phi<tPhi'} and \eqref{Phi>ctPhi'} that
\begin{align*}
\Phi(t) \ge c t\phi(t) = ct\min\{\phi_1(t), \phi_2(t)\} \ge c\min \{\Phi_1(t), \Phi_2(t)\} .
\end{align*}
Therefore 
\[
\frac{1}{\Phi(t) } \le \frac1c \max \left\{\frac{1}{\Phi_1(t)},  \frac{1}{\Phi_2(t)} \right\} 
\le \frac1c \left( \frac{1}{\Phi_1(t)} + \frac{1}{\Phi_2(t)} \right);
\]
integrating this inequality gives us \eqref{int1/Phi-01}. 
\end{proof}

\begin{lm}
\label{l:volb}
Let $\Psi:(0,1]\to \R_+$, $\Psi(1)>0$ be a decreasing function such that $\psi(s):=s\Psi(s)$ is increasing, $\psi(0_+)=0$, and such that
\[
\int_0 \frac{ds}{s\Psi(s)} <\infty. 
\]
 Assume that $t\mapsto \Psi(e^{-t})$ is convex near $\infty$. Then there exists an increasing function $t\mapsto t\alpha(t)>0$ on $[1, \infty)$ satisfying 
\begin{align}
\notag
\int_1^\infty \frac{dt}{t\alpha(t)} &< \infty \\
\intertext{and such that }
\label{bump-comp}
\alpha\left(\frac{\|f\|_*}{\|f\|_1} \right) \|f\|_* &\le \|f\|_{\Lambda_\psi}. 
\end{align}
\end{lm}

\begin{proof}

First, we can assume without loss of generality that $t\mapsto \Psi(e^{-t})$ is convex for all $t\in [1, \infty)$. Indeed, let this function be convex for $t> a$. Defining 
\[
\Psi_1(s) : = \left\{ \begin{array}{ll} \Psi(s) , \qquad &s< e^{-a}:= b ;\\
\Psi(b_-) & s\ge e^{-a} \end{array}\right. , 
\]
we can immediately see that the function $t\mapsto \Psi_1(e^{-t})$ is convex for all $t\ge 1$. And replacing $\Psi$ by  $\Psi_1$ we get an equivalent norm on $\Lambda_\psi$.  

So, let us assume that $t\mapsto \Psi(e^{-t})$ is convex for all $t\ge 1$. 
Define 
\[
\alpha(t):= \frac{\Psi(e e^{-t})}{t}, \qquad \gamma(t):= t \alpha(t) = \Psi(e e^{-t}),   
\]
so for $\Psi_0(s) := \ln(e/s)$ we have 
\begin{align}
\label{gamma(Psi_0)}
\Psi(s) = \gamma(\Psi_0(s)). 
\end{align}

Change of variables $s=e^{1-t}$ gives us that 
\[
\int_1^\infty\frac{dt}{t\alpha(t)} = \int_1^\infty \frac{dt}{\Psi(e^{1-t})} = \int_0^1 \frac{ds}{s\Psi(s)} <\infty. 
\]

To prove \eqref{bump-comp} let us first notice that because of homogeneity we can assume without loss of generality that 
\[
\|f\|_1 = \int_0^\infty N(t) dt =1. 
\]
Then defining probability measure $\mu$ by $d\mu = N(t) dt$ we using \eqref{gamma(Psi_0)} can rewrite \eqref{bump-comp}  as 
\begin{align}
\label{bump-comp-01}
\gamma\left( \int_0^\infty \Psi_0(N(t)) d\mu(t) \right) \le \int_0^\infty \gamma\left(\Psi_0(N(t))\right) d\mu(t)  
\end{align}
(recall that $\Psi_0(s) = \ln(e/s)$). 

But the function $\gamma$ is convex, so \eqref{bump-comp-01} follows immediately from Jensen inequality. 
\end{proof}

\begin%
{proof}[Proof of Lemma \ref{l:Orl-Entr}]
Let $\Phi(t)/(t\ln t)$ be increasing for all $t\ge t_0\ge e^e$. 

Applying Lemma \ref{l:min_bump} to functions $\Phi_1=\Phi$ and $\Phi_2$, where $\Phi_2(t) = t \ln^2 t$ for  $t\ge t_0$, we get a function $\Phi_0$ such that $\int^\infty 1/\Phi_0 <\infty$, and such that $\Phi_0(t) \le t\ln^2 t$ (for  $t\ge t_0$). 

The function $\Phi_0(t)/(t\ln t)$ is not necessarily increasing, but $\Phi_0$ is equivalent to the function $\Phi\ti{min} :=\min\{\Phi_1, \Phi_2\}$, and $\Phi\ti{min}(t)/(t\ln t)$ is increasing (for  $t\ge t_0$) as minimum of increasing functions.  Since $\Phi\ti{min}(t)/(t\ln t)$ is increasing and $\Phi_0$ is equivalent to $\Phi\ti{min}$ near $\infty$,  we conclude that for all  $t\ge t_0$
\begin{align}
\label{Phi_0-bounds}
ct\ln t \le \Phi_0(t) \le C t \ln^2 t  .
\end{align}
Since $\Phi_0$ is doubling, \eqref{Phi<tPhi'} and \eqref{Phi>ctPhi'} imply that 
\begin{align}
\label{Phi'_0-bounds}
c \ln t \le \Phi_0'(t) \le C  \ln^2 t  \qquad \forall t\ge t_0 .  
\end{align}
Therefore, replacing $\Phi$ by $\Phi_0$ we can assume without loss of generality that 
\begin{align}
\label{Phi-Phi'-bounds}
ct\ln t \le \Phi(t) \le C t \ln^2 t  , \qquad c \ln t \le \Phi'(t) \le C  \ln^2 t   \qquad \forall t\ge t_0 .
\end{align}

Now let us recall Lemma \ref{l:Orl-Lor}. We construct a function $\Psi_0$ dominated by the function $\Psi$ from the lemma, but still such that $\int_0 \frac{ds}{s\Psi_0(s)}<\infty$. Recall that $\Psi$ in Lemma \ref{l:Orl-Lor} was given by 
\[
\Psi(\tilde s) := \Phi'(t) \qquad \text{for}\quad \tilde s=\frac{1}{\Phi(t)\Phi'(t)}. 
\]
We can see from \eqref{Phi-Phi'-bounds}
that 
\[
\frac{c}{t \ln^4 t} \le \frac{1}{\Phi(t)\Phi'(t)} \le \frac{C}{t \ln^2 t}. 
\]
So defining 
\[
\Psi_0(s) := \frac{\Phi\ti{min}(t)}{t} \qquad \text{for}\quad  s=\frac{c}{t \ln^4 t}.
\]
we get that $\Psi_0\le C\Psi$. Indeed, the functions $\Psi$ and $\Psi_0$ are decreasing, $\Phi\ti{min}(t)/t \le C\Phi(t)/t\le C\Phi'(t)$, and  we have two relationships: $s\le\tilde s$ and $\Psi_0(s) \le C\Psi(\tilde s)$. 

We can pick $c$ sufficiently small, so $s\in (0,1]$ correspond to $t\in [t_1, \infty)$, $t_1\ge t_0 \ge e^e$. 

It is also clear that $\psi_0(s):=s\Psi_0(s)$ is increasing, $\psi_0(0_+)=0$. Let us check that $1/(s\Psi_0(s))$ is integrable near $0$: since $-ds=(t^{-2} \ln^{-2} t - t^{-2}\ln^{-5} t) dt\le Ct^{-2} \ln^{-2} t dt$
\[
\int_0 \frac{ds}{s\Psi_0(s)} \le C\int^\infty \frac{t}{\Phi(t)} t \ln^4 t \frac{dt}{t^2\ln^4 t} =C\int^\infty \frac{dt}{\Phi(t)}<\infty. 
\]

Finally, we claim that $-\Psi_0(s)/\ln s$ is decreasing. 
For this we need to show that
\[
-\frac{\ln s}{\Psi_0(s)} = \frac{t \cdot(\ln t + 4 \ln\ln t - \ln c)}{\Phi\ti{min}(t)} 
\]
is a decreasing  function of $t$ (and so the increasing function of $s$). But the term $t\ln t/\Phi\ti{min}(t)$ is decreasing by the assumption, the second term is decreasing because $(\ln\ln t)/\ln t$ is decreasing for $t\ge e^e$, and the last term is decreasing if $c\le 1$, which we always can assume without loss of generality. 


Define $\f(t) := \Psi_0(e^{-t})$, $t>0$, $\f(0):=0$. Clearly, $\f$ is increasing function on $[0,\infty)$.  Change of variable in the integral shows that 
\begin{align}
\label{1/sPsi=1/fi}
\int_0 \frac{ds}{s\Psi_0(s)}<\infty \qquad \iff \qquad \int^\infty \frac{dt}{\f(t)}<\infty. 
\end{align}
The fact that $-\Psi_0(s)/\ln s$ is decreasing translates to the statement that $\f(t)/t$ is increasing. Then for the inverse $\f^{-1}$ we get that $\f^{-1}(\tau)/\tau$ is decreasing. Also, $\f^{-1}(\tau)$ is clearly increasing, so $\f^{-1}$ is a pseudoconcave function. Therefore, by Proposition \ref{p:LCM} the least concave majorant $\wt \f^{-1}$ of $\f^{-1}$ satisfies
\[
\frac12 \wt\f^{-1}(\tau) \le \f^{-1}(\tau) \le \wt\f^{-1}(\tau). 
\]
Therefore
\begin{align}
\label{fi-wtfi}
\wt\f(t) \le \f(t) \le \wt\f(2t). 
\end{align}
The function $\wt\f$ is convex. Note that 
\[
\int^\infty \frac{dt}{\wt\f(t)} = 2 \int^\infty \frac{dt}{\wt\f(2t)} \le 2\int^\infty \frac{dt}{\f(t)} <\infty.
\] 

Therefore the function $\wt\Psi$, $\wt\Psi(s) := \wt\f(-\ln s)$ satisfies the assumptions of Lemma \ref{l:volb}. Since $\wt\Psi\le \Psi_0$, applying Lemma \ref{l:volb} to $\wt\Psi$ give us the conclusion of Lemma \ref{l:Orl-Entr}.
\end{proof}

\section{An embedding theorem and two sided bumps for the Lerner type operators}
\label{s:lerner-2side}

In this section it is more convenient to define $\bu\ci I^*$ as
\[
\bu\ci I^* := \|M(u\1\ci I)\|\ci{L^1(I)} = \La M(u\1\ci I) \Ra\ci I \le\|u\|\ci{\Lambda_{\psi_0}(I)} =  \|u\1\ci I\|\ci{\Lambda_{\psi_0}(I)}\approx \|u\|\ci{L\log L(I)}, 
\]
where $M=M\ci\cD$ is the martingale maximal function, 
\[
M  f (x) = \sup_{I\in\cD:\, x\in I} |I|^{-1} \int_i |f| dx, 
\]
and $\psi_0(s):= s\ln(e/s)$, and $\Lambda_{\psi}(I) = \Lambda_\psi (I, \frac{dx}{|I|})$ is the Lorentz space.

Recall that a sequence $a=\{a\ci I\}\ci{I\in\cD}$ is called \emph{Carleson} if 
\[
\sup_{I_0\in\cD} |I_0|^{-1}\sum_{I\in\cD} |a\ci I| \cdot |I| =: \|a\|\ci{\text{Carl}}<\infty
\]
\begin{thm}
\label{t:lern-bump-02}
Let $a=\{a\ci I\}\ci{I\in \cD:}$ be a 
Carleson sequence and let $T$ be the operator 
\[
T f = \sum_{I\in\cD} \La f \Ra\ci I a\ci I \1\ci I. 
\]
Let $\alpha:[1,\infty )\to \R_+$ be an increasing function satisfying $\int_1^\infty \frac{dt}{t\alpha(t)} <\infty$, and let 
\begin{align}
\label{2-sided-H1-bump}
\alpha\left(\bu^*\ci I/\bu\ci I\right) \bu\ci I^* \bv\ci I^* \alpha\left(\bv^*\ci I/\bv\ci I\right) \le A <\infty
\end{align}
Then the operator $f\mapsto T(fu)$ is a bounded operator $L^2(u)\to L^2(v)$, 
\[
\| T (fu)\|\ci{L^2(v)} \le C \|a\|\ti{Carl} A^{1/2}\|f\|\ci{L^2(u)}. 
\]
where $C=4C_\alpha$, 
\begin{align}
\label{C_alpha-01}
C_\alpha =  \frac{1}{\alpha(1)} + \int_1^\infty \frac{dt}{t\alpha(t)}. 
\end{align}
\end{thm}

\begin{rem*}
One can state condition \eqref{2-sided-H1-bump} with 2 different functions $\alpha_1(\bu^*\ci I/\bu\ci I)$ and $\alpha_2(\bv^*\ci I/\bv\ci I)$, such that $\int_1^\infty \frac{dt}{t\alpha_k(t)}<\infty$ for $k=1,2$. However defining $\alpha(t) =\min\{\alpha_1(t), \alpha_2(t)\}$  one can easily see that $\alpha$ is an increasing function satisfying $\int_1^\infty \frac{dt}{t\alpha(t)} <\infty$. So,  replacing $\alpha_{1,2}$ with $\alpha$ we get a weaker condition, which is exactly \eqref{2-sided-H1-bump}. Thus, there is nothing to gain considering different function $\alpha$. 
\end{rem*}

Theorem \ref{t:lern-bump-02} follows immediately from the following ``embedding theorem''  via Cauchy--Schwarz. 

\begin{thm}
\label{t:embed-03}
Let $a=\{a\ci I\}\ci{I\in\cD} $, $a\ci I\ge 0$ be a Carleson sequence. Then for any $f\in L^2(u)$
\[
\sum_{I\in\cD} \frac{|\La fu\Ra\ci I|^2}{\alpha(\bu^*\ci I/\bu\ci I) \bu^*\ci I} a\ci I |I| \le C \|a\|^{\phantom 2}\ti{Carl} \|f\|\ci{L^2(u)}^2, 
\]
where $C=4C_\alpha$, 
\[
C_\alpha =  \frac{1}{\alpha(1)} + \int_1^\infty \frac{dt}{t\alpha(t)}. 
\]
\end{thm}

To prove this theorem we can assume without loss of generality that $\|a\|^{\phantom 2}\ti{Carl}=1$. Let us introduce the following weighted \emph{Carleson potential} 
\[
P(J) = \mu(J)^{-1}\sum_{I\in\cD:I\subset J} a\ci{I} \mu(I), 
\]
where $d\mu = u dx$, so $\mu(I) =\La u\Ra\ci I |I|$. Note, that $\mu$ here is not the underlying measure, 
$\mu(A) =\int_A u dx$, and the underlying measure of $A$ is denoted as $|A|$!

If $P(J)\le K<\infty$, then the martingale Carleson Embedding theorem implies that 
\[
\sum_{I\in\cD} 
\frac{|\La fu\Ra\ci I|^2}{ \La u\Ra\ci I} a\ci I |I|
 =
 \sum_{I\in\cD} \left| \mu(I)^{-1} \int_I f d\mu \right|^2 a\ci I \mu(I) 
\le 4 K \|f\|_{L^2(\mu)}^2.  
\]

Unfortunately, the potential $P$ is unbounded, so the above estimate is not true under the assumptions of the theorem: we need to put something bigger in the denominator to get a true estimate.  And the proof becomes more involved. 

Namely, we can only guarantee that 
\begin{align}
\label{P_le_U^*/u}
P(I) =\mu(I)^{-1}\sum_{I'\in \cD, I'\subset I} a\ci{I'} \La u \Ra\ci{I'} |I'| \le \mu(I)^{-1} \|a\|^{\phantom 2}\ti{Carl} \La M \1\ci{I} u\Ra\ci{I} |I| \le  \|a\|\ti{Carl} \bu^*\ci I/\bu\ci I;
\end{align}
the first inequality here is the standard estimate of the Carleson embedding via maximal function. 

\subsection{Embedding theorem via  Bellman function}
\label{s:Bell_for_embedd-01}

To prove the embedding theorem we will use the Bellman function technique. Namely, suppose we constructed a function $\cB(x, y)$, $x\in \R$, $y\ge 0$, such that 

\begin{enumerate}
\item $\cB $ is convex;
\item $\displaystyle -\frac{\p\cB}{\p y} \ge 
x^2 \f(y)$, 
where 
\[
\f(y) := \left\{ \begin{array}{ll} {1}/(y\alpha(y)),  \qquad & y\ge 1, \\ 
 1/\alpha(1), & 0<y<1; \end{array} \right.
\]
\item $0\le \cB(x, y)\le C x^2$, where $C=4C_\alpha$ 
\end{enumerate}
Then the embedding theorem is proved. 

Namely, define
\[
x\ci I := \frac{\La fu\Ra\ci I}{ \La u \Ra \ci I} := \frac{\bff\ci I }{ \bu\ci I} , \qquad 
y\ci I = P(I),  
\]
and let $B\ci I := \cB(x\ci I, y\ci I)$. 

Note that the averages 
\[
\overline x\ci I := \mu(I)^{-1} \sum_{I'\in\ch(I)} x\ci{I'} \mu(I'), \qquad \overline y\ci I:= \mu(I)^{-1} \sum_{I'\in\ch(I)} y\ci{I'} \mu(I')
\]
satisfy $\overline x\ci I = x\ci I$ and 
\begin{align}
\label{y-y_av}
y\ci I - \overline y\ci I = a\ci I . 
\end{align}
Therefore, using convexity of $\cB$ and then mean value theorem together with the property \cond2 and \eqref{y-y_av} we get that
\[
- B\ci I + \sum_{I'\in\ch(I) } \frac{\mu(I')}{\mu(I)} B\ci{I'}  \ge \cB(\overline x\ci I, \overline y\ci I) -B\ci I
= \cB(x\ci I, \overline y\ci I) -B\ci I \ge \left(\frac{\bff\ci I}{\bu\ci I}\right)^2\f(y\ci I) a\ci I .   
\]
Note that the mean value theorem gives us the estimate with $\f(y)$ for some $y\in (\overline y\ci I , y\ci I)$; since $t\mapsto t\alpha(t)$ is increasing, $\f$ is decreasing, and we can replace $\f(y)$ by $\f(y\ci I)$.
Moreover, since $y\ci I=P\ci I \le \bu^*\ci I/\bu\ci I$ (see \eqref{P_le_U^*/u}) we have $\f(y\ci I)\ge \f(\bu^*\ci I/\bu\ci I)$, so we can replace $\f(y\ci I)$ by $\f(\bu^*\ci I/\bu\ci I)$.    Multiplying by $\mu(I)$ we can rewrite this new estimate as 
\[
 \sum_{I'\in\ch(I) } \mu(I')  B\ci{I'}\ - \mu(I) B\ci I  \ge\frac{\bff\ci I^2}{\bu\ci I \alpha(\bu^*\ci I/\bu\ci I) \bu^*\ci I/\bu\ci I} a\ci I |I|
\]
Writing this estimate for $I=I_0$ and going $n-1$ generations down we get using $\cB\ge0$
\[
\sum_{\substack{I\in \ch_k(I_0),\\ 0\le k <n }} \frac{\bff\ci I^2}{\alpha(\bu^*\ci I/\bu\ci I) \bu^*\ci I} a\ci I |I| 
\le \sum_{I\in \ch_n(I_0)} \mu(I) B\ci I - \mu(I_0) B\ci{I_0} \le  \sum_{I\in \ch_n(I_0)} \mu(I) B\ci I  .
\]
The estimate $\cB(x, y) \le C x^2$ implies 
\[
\mu(I) B\ci I \le C \frac{\bff\ci I^2}{\bu} |I| \le C \int_I f^2 u dx;
\]
the last inequality is just the Cauchy--Schwarz. Therefore
\[
\sum_{k=0}^{n-1}\ \sum_{I\in \ch_n(I_0)} \frac{\bff\ci I^2}{\alpha(\bu^*\ci I/\bu\ci I) \bu^*\ci I} a\ci I |I|  \le C \int_{I_0} f^2 u dx,  
\]
and letting $n\to \infty$ we get that 
\[
\sum_{I\in\cD, I\subset I_0} \frac{\bff\ci I^2}{\alpha(\bu^*\ci I/\bu\ci I) \bu^*\ci I} a\ci I |I| \le C \int_{I_0} f^2 u dx. 
\]
Summing the above estimate for all $I_0\in \cD_m$  and letting $m\to -\infty$, we get the conclusion of Theorem \ref{t:embed-03}. \hfill \qed

\subsection{Constructing the Bellman function}
\label{s:constrBell-01}
We will look for a function $\cB:\R\times \R_+ \to [0,\infty)$ of form $\cB(x,y)= x^2 m(y)$, where $m$ is a $C^1$ convex function. The Hessian of $\cB$ is easy to compute: 
\[
\left(\begin{array}{cc}
\cB_{xx} & \cB_{xy} \\ 
\cB_{xy} & \cB_{yy} \\ 
\end{array}\right)
=
\left(\begin{array}{cc}
2m(y) & 2 x m'(y) \\ 
2 x m'(y) & x^2 m''(y) \\ 
\end{array}\right) \,.
\]
It is clear that the Hessian is positive semidefinite (i.e.~the function $\cB$ is convex) if and only if 
\begin{align}
\label{conv-02}
\left( \begin{array}{cc}
2m(y) & 2 m'(y) \\ 
2 m'(y) &  m''(y) \\ 
\end{array} \right) 
\ge 0 \,, 
\end{align}
(if $m''(x) = \infty$ the \eqref{conv-02} is automatically satisfied),  
or, equivalently (if we assume that $m(y)>0$ for all $y$)
\[
2(m')^2\le m m''. 
\]

Note, that property \cond2 of $\cB$ is equivalent to the estimate 
\begin{align}
\label{est_m'}
-m'(t) \ge \frac1{t\alpha(t)} =: \f(t). 
\end{align}

\begin{lm}
\label{l:m'}
Let $\f : \R_+ \to [0,\infty)$ be a bounded decreasing function such that 
\[
\|\f\|_1 = \int_0^\infty \f(t) dt <\infty. 
\]
Then there exists a bounded decreasing convex $C^1$-function $m : [0,\infty)\to [0,\infty)$, $m(0)\le 4 \|\f\|_1$,  satisfying \eqref{conv-02} 
 and such that 
\[
-m'(t)\ge \f(t) \qquad \forall t>0 .
\] 
Moreover, $-m'(0_+)\le 4 \f(0_+)= 4\lim_{t\to 0_+} \f(t)$. 
\end{lm}

\begin{proof}
Consider first the case when $\f=\1\ci{[0,1]}$. Then the function $m=m_1$, $m_1(y) = 4/(1+y)$ satisfies \eqref{conv-02} and 
\begin{align}
\label{m1'}
-m_1'(y) = 4(1+y)^{-2} \ge \1\ci{[0,1]}(y) \,.
\end{align}
Defining $m_r(y) := m_1(y/r) = 4/(1+y/r)$, $r>0$ we get that 
\begin{align}
\label{m_r'>1}
-m_r'(y) = 4 r^{-1} (1+y/r)^{-2} \ge r^{-1}\1\ci{[0,1]} (y/r) = r^{-1}\1\ci{[0,r]} (y). 
\end{align}

Note, that  any function $\f$ from the lemma can be dominated by a multiple of a function from the convex hull of function $r^{-1}\1_{[0,r]}$. Namely, there exists a non-negative measure $\mu$ on $\R_+$ of total mass $\|\f\|_1$ such that for  all $y>0$
\begin{align}
\label{phi-int-repr}
\f(y) \le \int_{\R_+} r^{-1}\1\ci{[0,r]}(y) d\mu(r)
\end{align}
and for almost all $y>0$ there is the equality (for an upper semi-continuous $\f$ the equality holds everywhere). 

Since the property \eqref{conv-02} is preserved for linear combinations we get integrating \eqref{m_r'>1} that the function $m$
\[
m (y) := \int_{\R_+} m_r(y) d\mu(r) 
\]
satisfies the conclusion of the lemma. 

To justify the changing the order of integral and derivative when integrating \eqref{m_r'>1} we can first define 
\begin{align}
\label{def_m'}
m' (y) : = \int_{\R_+} m_r'(y) d\mu(r)
\end{align}
and then conclude that by Tonelli and the fundamental Theorem of Calculus 
\begin{align*}
m(y) = \int_{y}^\infty -m'(x) dx  & = \int_{y}^\infty \int_{\R_+} -m_r'(x) d\mu(r) dx 
\\
& = 
\int_{\R_+} \int_{y}^\infty  -m_r'(x) dx d\mu(r) = 
\int_{\R_+} m_r(y) d\mu(r). 
\end{align*}

To get the estimate $-m'(0_+)\le 4 \f(0_+)$ we notice that the monotone convergence theorem and the fact that there is an equality for almost all $y>0$ in \eqref{phi-int-repr}  imply that 
\[
\f(0_+) = \int_{\R_+} r^{-1} d\mu(r) . 
\] 
On the other hand, applying the monotone convergence theorem to \eqref{def_m'} (with sign ``$-$'') as $y\searrow 0$ and recalling that $-m_r'(y) = 4 r^{-1} (1+y/r)^{-2}$ we get that
\[
-m'(0_+) = 4 \int_{\R_+} r^{-1} d\mu(r) .
\]
\end{proof}

Applying Lemma \ref{l:m'} let us pick a bounded decreasing function $m$ satisfying \eqref{conv-02} such that 
\begin{align}
\label{m'-maj}
-m'(t) \ge \f(t), \qquad \f(t) =
\left\{
\begin{array}{ll} 1/(t\alpha(t)), \qquad & t\ge 1 \\
1/\alpha(1), & 0<t\le 1. 
\end{array}\right.
\end{align}
Note, that by Lemma \ref{l:m'}
\[
m(t) \le m(0_+) \le 4 \|\f\|_1 = 4 C_\alpha. 
\]
So, we constructed the function $\cB$ satisfying the conditions \cond1--\cond3 with $C=4C_\alpha$, and the embedding theorem is proved.

\begin{rem}
\label{r:conv-m'}
As one can see from the proof of Lemma \ref{l:m'}, the function $-m'$ is convex. Indeed, one can see from \eqref{m1'} that the function $-m_1'$  is convex, and so are the functions $-m_r'$. Together with the identity \eqref{def_m'} it immediately implies the convexity of $-m'$. 

We do not need this fact right now, we will need it in later. 
\end{rem}

\section{Concavity properties of 
\texorpdfstring{$\bu^*$}{u*<sub>I}}
\label{s:ConvOf_n}

For treating two-sided bumping for Haar shifts is is more convenient to consider $\bu^*\ci I:=\|u\|\ci{\Lambda_{\psi_0}(I)}$, $\psi_0(s)=s\ln(e/s)$.

We want to investigate the behavior of quantity $\bu^*\ci I:=\|u\|\ci{\Lambda_{\psi_0}(I)}$, $\psi_0(s)=s\ln(e/s)$, as we go from $I\in\cD$ to its children, we want to show that it has a supermartingale (concave) behavior, and give a quantitative characteristic of this concavity.  Recall that for $u\ge0$ any Lorentz norm $\|u\|\ci{\Lambda_{\psi}(I)}$ can be computed using the normalized distribution function $N_I^u$, 
\begin{align}
\label{NormDistrFn}
N_I^u (t):= |I|^{-1}\left|\left\{ x\in I : u (x)>t\right\}\right|, \qquad t\ge 0, 
\end{align}
namely 
\[
\|u\|\ci{\Lambda_\psi(I)} := \int_0^1 \psi(N_I(t)) dt . 
\]
The distribution functions $N_I^u$ possess very simple martingale behavior, 
\begin{align}
\label{mart-N}
N_{I_0}^u = \sum_{I\in\ch(I_0)} \frac{|I|}{|I_0|} N_I^u, 
\end{align}
so it is convenient to treat $\bu^*$ as a functional on normalized  distribution functions $N$, i.e.~on decreasing functions $N:\R_+\to [0,1]$, 
\begin{align}
\label{u^*-N}
\bu^*(N) := \int_0^1 \psi_0(N(t)) dt, \qquad \psi_0(s) =s\ln(e/s). 
\end{align}
We will also need the functional $N\mapsto \bu(N)$, 
\begin{align}
\label{u-N}
\bu(N):=\int_0^\infty N(t) dt. 
\end{align}
Clearly if $N=N_I^u$, then $\bu(N_I^u) = \La u \Ra\ci I =:\bu\ci I$.


Let $N=N_0$ and $N_1$ be two distribution functions, and let $\sd N:=N_1 -N$. We want to compute the second derivative of the function $\theta \mapsto \bu^*(N+\theta\sd N)$. 

Let $N_\theta :=N + \theta\sd N$, and let 
\[
\bu_\theta := \int_0^\infty N_\theta(t) dt. 
\]
If we think of the function $N_\theta$ as of the distribution function of a function $u_\theta$ on, say, $[0,1]$, then $\bu_\theta$ is the average of the function $u_\theta$. 
Also, denote
\begin{align}
\label{delta-u}
\sd \bu := \bu_1-\bu_0 = \int_0^\infty \sd N(t) dt, \qquad \bu_\sd := \int_0^\infty | \sd N(t)| dt
\end{align}

\begin{lm}
\label{l:conv-01}
Let $N$, $N_1$ be compactly supported distribution functions taking finitely many values. Then 
\[
-\frac{d^2 \bu^*(N_\theta)}{d\theta^2} 
\ge \frac{(\bu_\sd)^2}{\bu_\theta} \ge \frac{|\sd\bu|^2}{\bu_\theta} 
\]
\end{lm}

\begin{proof}
Recall that 
\[
\bu^*(N_\theta) = \int_0^\infty \psi (N(t)+\theta\sd N(t)) \,dt. 
\]
Since 
$\psi''(s) =-1/s$, we get, differentiating under the integral that 
\begin{align}
\label{d2n}
- \frac{d^2\bu^*(N_\theta)}{d\theta^2} & = \int_0^\infty \frac{\sd N(t)^2}{N_\theta(t)} dt .
\end{align}
Note that under the assumptions of the lemma there is no problem in justifying differentiating the integral.

The first inequality follows from the Cauchy--Schwartz:
\begin{align*}
(\bu_\sd)^2 = \left(\int_0^\infty |\sd N(t)| \,dt \right)^2 
& \le \left(\int_0^\infty \frac{\sd N(t)^2}{N_\theta(t)} dt \right) \left( \int_0^\infty N_\theta (t) dt \right) \\
& = 
\left(\int_0^\infty \frac{\sd N(t)^2}{N_\theta(t)} dt \right) \bu_\theta. 
\end{align*}
The second inequality follows trivially because $|\sd \bu | \le \bu_\sd$. 
\end{proof}

\begin{cor}
\label{c:conv-02}
Let $N$, $N_1$, $N_2$ be the distribution functions such that $N=(N_1+N_2)/2$ and $\bu(N_{1,2})<\infty$. Denote $\sd N:= N_1-N$ and let $\sd \bu$ and $\bu_\sd$ be defined by \eqref{delta-u}. Then 
\[
\bu^*(N) - \frac{\bu^*(N_1)+\bu^*(N_2)}2 \ge \frac12\cdot \frac{(\bu_\sd)^2}{\bu} \ge \frac12 \cdot \frac{(\sd\bu)^2}{\bu}
\] 
\end{cor}

\begin{proof}
It is sufficient to prove this corollary only for compactly supported distribution functions taking finitely many values: the general case is then obtained by approximation. 

So, let $N_1$, $N_2$ be compactly supported distribution functions, taking finitely many values. Introducing function 
\[
F(\tau) = \bu^*(N) - \frac{\bu^*(N+\tau \sd N) + \bu^*(N-\tau \sd N)}2
\] 
and noticing that $F'(0)=0$ we get using Taylor's formula that 
\[
\bu^*(N) - \frac{\bu^*(N_1)+\bu^*(N_2)}2  = F(1) - F(0) = \frac{F''(\theta)}2
\]
for some $\theta\in (0,1)$. Recalling that by Lemma \ref{l:conv-01}
\[
F''(\theta) \ge \frac{(\bu_\sd )^2}{2} \left( \frac{1}{\bu_\theta} + \frac{1}{\bu_{-\theta}}\right)
\]
and noticing that by convexity 
\[
\frac{1}{2} \left( \frac{1}{\bu_\theta} + \frac{1}{\bu_{-\theta}}\right) \ge \frac{1}{\bu}
\]
we get the conclusion. 
\end{proof}

\section{Bumps for the Haar shifts and paraproducts}

For a weight $u$ let $\bu\ci I:= \La u\Ra\ci I = \|u\|\ci{L^1(I)}$, $\bu^*\ci I:= \| u \|\ci{\Lambda_{\psi_0}(I)}$

\begin{thm}
\label{t:H1_bumps-01}
Let $\alpha:[1, \infty)\to \R_+$ be an increasing function such that 
\begin{align*}
C_\alpha:=
\frac{1}{\alpha(1)} + \int_1^\infty \frac1{t\alpha(t)} dt <\infty. 
\end{align*}
Let $u$ be a 
weight. 

Then for any $f\in L^2(u)$
\begin{align}
\label{embed-02}
\sum_{I\in\cD} \frac{\|\Delta\ci I (f u)\|\ci{L^1(I)}^2}{\alpha(\bu^*\ci I/ \bu\ci I )\bu^*\ci I } |I| 
\le 36 C_\alpha\|f\|\ci{L^2(u)}^2, 
\end{align}
\end{thm}

Using this theorem and Cauchy--Schwarz we get Theorem \ref{t:main-01} for Haar shifts with $C=(36 A^{1/2} C_\alpha)^2$, where $A$ is the supremum in \eqref{entr-bump-01}. 

We will prove  Theorem \ref{t:H1_bumps-01} using the Bellman function method. By the homogeneity we can assume without loss of generality that $C_\alpha=1$.

\subsection{The Bellman function and the main dyadic inequality}
Let $\cB=\cB_\alpha$ be the function constructed in Section \ref{s:Bell_for_embedd-01} above. 



Define function $\wt \cB:\R\times \cN\to [0,\infty)$ (recall that $\cN$ is the set of all compactly supported distribution functions taking finitely many values) by 
\begin{align*}
\wt\cB(\bff, N) 
=2\bu(N)\cB\left(\frac{\bff}{\bu(N)}, \frac{\bu^*(N)}{\bu(N)}\right)
+ \frac{\bff^2}{\bu(N)} = :
2\wt\cB_1(\bff, N) + \wt\cB_2(\bff, N). 
\end{align*}
It follows from property \cond3 of $\cB$ that 
\begin{align}
\label{bound_B}
0\le \wt\cB(\bff, N) \le 9 \frac{\bff^2}{\bu(N)}. 
\end{align}
(recall that we assumed that $C_\alpha=1$).

Recalling that $\cB(x,y) = x^2 m(y)$, where $m$ is the function obtained by applying Lemma \ref{l:m'} to the function  $\f$
\[
\f(t) =
\left\{
\begin{array}{ll} 1/(t\alpha(t)), \qquad & t\ge 1, \\
1/\alpha(1), & 0<t\le 1 \,, 
\end{array}\right.
\]
we can write $\wt\cB_1(\bff, N) = \cB_1(\bff, \bu(N), \bu^*(N))$, where $\cB_1$ is a function of 3 scalar arguments, 
\begin{align}
\label{cB_1}
\cB_1(\bff, \bu, \bu^*) = \frac{\bff^2}{\bu} m (\bu^*/\bu). 
\end{align}

%

\begin{lm}
\label{l:dyBellMainIneq-01}
Let 
\begin{align*}
\bff = \frac{\bff_+ + \bff_-}{2}, \qquad N(t) = \frac{N_+(t) + N_-(t)}{2}.
\end{align*}
Then the fuction $\wt\cB$ introduced above satisfies
\begin{align}
\label{d-MainIneq}
\frac12 \Bigl( \wt\cB(\bff_+, N_+)   + \wt\cB(\bff_-, N_-) \Bigr) - \wt\cB(\bff, N) 
\ge \frac{1}{2}   \cdot  \frac{ (\bff_+ -\bff)^2}{\alpha(\bu^*/\bu ) \bu^*},  
\end{align}
where $\bu=\bu(N)$, $\bu^*=\bu^*(N)$.
 (Note that $\bff_+ -\bff = \bff-\bff_-$, so we can replace $(\bff_+ -\bff)^2$ in the right side by $(\bff_- -\bff)^2$)
\end{lm}

To prove the above lemma we need the following fact. 

\begin{lm}
\label{l:convB-03}
Let again  $\bu=\bu(N)$, $\bu^*=\bu^*(N)$, and let 
\[
\sd\bu := \int_0^\infty \sd N(t) dt
\]
If $N_{\pm} = N\pm\sd N$, $\bff=(\bff_++\bff_-)/2$ then
\[
\frac{\wt\cB_1(\bff_+, N_+) + \wt\cB_1(\bff_-, N_-)}{2} - \wt\cB_1(\bff, N)\ge 
\frac12 \cdot \frac{(\sd\bu)^2 \bff^2}{\bu^2\alpha(\bu^*/\bu)\bu^*}
. 
\]
\end{lm}

\begin{proof}
Denote 
\[
\bu_{\pm} = \bu(N \pm \sd N), \qquad \bu^*_{\pm} = \bu^*(N \pm\ \sd N), 
\]
 and let 
 \[
 \bu^*_0 := (\bu^*_+ +\bu^*_{-})/2, \qquad \bu_0 = (\bu_+ +\bu_{-})/2;
\]
note that $\bu_0=\bu$, but generally we can only say that $\bu^*_0\le \bu^*$. Note also that $\bu_{\pm} = \bu\pm \sd \bu$.

Recall that $\wt\cB_1(\bff, N) = \cB_1(\bff, \bu, \bu^*)$ (see \eqref{cB_1}), so 
\begin{align}
\label{conv-wtcB_1}
\wt\cB_1(\bff_+, N_+) &+ \wt\cB_1(\bff_{-}, N_{-}) - 2 \wt\cB_1(\bff, N) 
\\ \label{conv-wtcB_1-02} &= 
\wt\cB_1(\bff_+, N_+) + \wt\cB_1(\bff_{-}, N_{-}) -2 \cB_1(\bff, \bu, \bu^*_0) 
\\ \label{conv-wtcB_1-03} & \qquad \qquad 
+ 2 \left( \cB_1(\bff, \bu, \bu^*_0) - \cB_1(\bff, \bu, \bu^*) \right)
\end{align}

Denoting $\sd^2 \bu^*:=\bu^*_0-\bu^*$ we can estimate  
estimate the term \eqref{conv-wtcB_1-03} by applying mean value theorem to \eqref{cB_1}:
\[
\cB_1(\bff, \bu, \bu^*_0) - \cB_1(\bff, \bu, \bu^*) = \frac{\bff^2}{\bu} m'\left(\frac{\bu^* +\theta\sd^2\bu^*}{\bu}\right) \frac{\sd^2\bu^*}{\bu}\,,
\]
where $0<\theta<1$. By Corollary \ref{c:conv-02} 
\[
-\sd^2 \bu^* \ge \frac12\cdot \frac{(\sd\bu)^2}{\bu}, 
\]
so
\begin{align}
\notag
\cB_1(\bff, \bu, \bu^*_0) - \cB_1(\bff, \bu, \bu^*) &\ge - \frac12 \cdot \frac{(\sd\bu)^2 \bff^2}{\bu^3} m'\left(\frac{\bu^* +\theta\sd^2\bu^*}{\bu}\right) 
\\
\notag
& \ge 
- \frac12 \cdot \frac{(\sd\bu)^2 \bff^2}{\bu^3} m'\left(\frac{\bu^*}{\bu}\right) \,;
\end{align}
the last inequality holds because $-m'$ is decreasing ($m$ is convex) and $\sd^2\bu^*\le 0$ 
(recall also that $m'<0$). Recalling that $-m'(t) \ge 1/(t\alpha(t))$ we get from there
\begin{align}
\label{deltaB_1}
\cB_1(\bff, \bu, \bu^*_0) - \cB_1(\bff, \bu, \bu^*) \ge \frac12 \cdot \frac{(\sd\bu)^2 \bff^2}{\bu^2\alpha(\bu^*/\bu)\bu^*}
\end{align}

Next, we want to show that the term \eqref{conv-wtcB_1-02} is non-negative. To do that we use the convexity of the function $\cB(x, y) = x^2 m(y)$, defined in Section \ref{s:Bell_for_embedd-01}.

Denoting
\begin{align*}
x & = \bff/\bu, \qquad & y & =\bu^*_0/\bu \\
x_{\pm} & = \bff_{\pm}/\bu_{\pm} & y_{\pm} & = \bu^*_{\pm}/\bu_{\pm}
\end{align*}
we can write \eqref{conv-wtcB_1-02} as
\begin{align}
\label{diff2-cB^1}
\bu_{+} \cB^1(x_{+}, y_{+} ) + \bu_{- } \cB^1(x_{-}, y_{-} ) -2\bu \cB^1(x,y)
\end{align}
Since $\bu = (\bu_+ + \bu_{-})/2$, we have
\[
x = \frac{\bu_{+}}{2\bu} x_+ + \frac{\bu_{-}}{2\bu} x_{-} , \qquad
y = \frac{\bu_{+}}{2\bu} y_+ + \frac{\bu_{-}}{2\bu} y_{-}, 
\]
so the convexity of $\cB^1$ implies that 
\[
\frac{\bu_{+}}{2\bu}  \cB^1(x_{+}, y_{+} )  + \frac{\bu_{-}}{2\bu}  \cB^1(x_{-}, y_{-} ) - 
\cB^1(x,y) \ge 0.
\]
But this means that \eqref{diff2-cB^1}, and so \eqref{conv-wtcB_1-02} are non-negative. 

Combining the estimate \eqref{deltaB_1} for \eqref{conv-wtcB_1-03} with the non-negativity of  \eqref{conv-wtcB_1-02}   we get the conclusions of the lemma. 
%
\end{proof}

\begin{proof}[Proof of Lemma \ref{l:dyBellMainIneq-01}]
The inequality probably can be verified by an elementary algebra, but the following geometric proof seems to be more illuminating. 

Denote $\bu_\pm :=\bu (N_\pm)$. 

Let $I=[0,1)$ be the unit interval, and let $I_+$ and $I_-$ be its right and left halves respectively. Let $\mu$ be a measure on $I$ such that $\mu(I_+)=\bu_+/2$, $\mu(I_-)=\bu_-/2$, so $\mu(I) = \bu$. Let $f$ be a function on $I$, 
\[
f := \frac{\bff_+}{\bu_+} \1\ci{I_+}+ \frac{\bff_-}{\bu_-}\1\ci{I_-} \,.
\]
Then 
\[
\mu(I)^{-1} \int_I f d\mu = \frac{\bff}{\bu}
\]
and the function $\frac{\bff}{\bu}\1\ci I$ is the orthogonal projection in $L^2(\mu)$ of $f$ onto constants. 
Then we get using the Pythagorean theorem  than
\begin{align}
\label{Pyth-01}
\frac12 \left(\frac{\bff_+^2}{\bu_+} + \frac{\bff_-^2}{\bu_-} \right) - \frac{\bff^2}{\bu}
=\| f\|\ci{L^2(\mu)}^2 - \left\| \frac{\bff}{\bu}\1\ci I \right\|_{L^2(\mu)}^2  = \left|(f, h^\mu)\ci{L^2(\mu)} \right|^2
\end{align}
where $h^\mu=h\ci I^\mu$ is and $L^2(\mu)$ Haar function of $I$, i.e a function which is constant on intervals $I_\pm$, normalized by $\|h^\mu\|\ci{L^2(\mu)}=1$ and is orthogonal to constants, 
\[
\int_I h^\mu d\mu=0    
\]
(clearly, such a function is unique up to a constant unimodular factor). 

Let $h =h\ci I:= \1\ci{I_+} - \1\ci{I_-}$ be the non-weighted Haar function. Then
\[
(f, h)\ci{L^2} =\frac12\left(\bff_+ - \bff_-\right) = (\bff_+- \bff) =(\bff -\bff_-). 
\]
Let $\wt h^\mu := h - a_\mu \1\ci I$, where 
\[
a_\mu = \frac{\bu_+ - \bu}{\bu} =\frac{\sd \bu}{\bu}
\]
is the unique constant which makes $\wt h^\mu$ orthogonal to constants in $L^2(\mu)$, i.e.~such that $\int_I \wt h^\mu d\mu=0$. Note that $\wt h^\mu$ is a constant multiple of the normalized $L^2(\mu)$-Haar function $h^\mu$. 

Direct computations show that 
\[
\| \wt h^\mu\|\ci{L^2(\mu)}^2 = \frac{\bu_+ \bu_-}{\bu^2} \bu \le \bu = \bu\|h^\mu\|\ci{L^2(\mu)}^2. 
\]
so
\begin{align}
\label{wth-01}
\left| (f, \wt h^\mu)\ci{L^2(\mu)}\right| \le \bu^{1/2} \left| (f,  h^\mu)\ci{L^2(\mu)}\right|
\end{align}
($\wt h^\mu$ is a constant multiple of $h^\mu$). We can write 
\begin{align*}
|\bff_+ - \bff| = \left| (f, h)\ci{L^2(\mu)} \right| = \left| (f, \wt h^\mu + a_\mu \1\ci I)\ci{L^2(\mu)} \right| 
\le \left| (f, \wt h^\mu)\ci{L^2(\mu)} \right| + |a_\mu| \bff. 
\end{align*}
Using the inequality $(a+b)^2 \le 2(a^2+b^2)$ and taking \eqref{wth-01} and the fact that $|a_\mu|\le 1$ into account we get
\[
\frac{|\bff_+ - \bff|^2}{2\bu} \le \left| (f,  h^\mu)\ci{L^2(\mu)} \right|^2  + \frac{\bff^2(\sd\bu)^2}{\bu^3}  \,.
\]
Recalling \eqref{Pyth-01} we get from here
\begin{align}
\label{mainDyPrelim}
\frac{|\bff_+ - \bff|^2}{2\bu} \le   \frac12 \left( 
\frac{\bff_+^2}{\bu_+} + \frac{\bff_-^2}{\bu_-} \right) - \frac{\bff^2}{\bu}  + \frac{(\sd\bu)^2 \bff^2}{\bu^3}  \,.
\end{align}
The assumption  $C_\alpha=1$ implies that $t\alpha(t) \ge 1$ for $t\ge 1$, so dividing left hand side and the last term in the right hand side of \eqref{mainDyPrelim} by $\alpha(\bu^*/\bu)\bu^*/\bu$ we get
\begin{align}
\label{mainDyPrelim-01}
\frac{|\bff_+ - \bff|^2}{2\alpha(\bu^*/\bu)\bu} \le   \frac12 \left( 
\frac{\bff_+^2}{\bu_+} + \frac{\bff_-^2}{\bu_-} \right) - \frac{\bff^2}{\bu}  + \frac{(\sd\bu)^2 \bff^2}{\bu^2 \alpha(\bu^*/\bu)\bu^*}  \,.
\end{align}

By the definition of $\wt\cB_2$ 
\[
\frac12 \left(\wt \cB_2(\bff_+, N)  \wt \cB_2(\bff_-, N_-) \right)-  \wt \cB_2(\bff, N)
= \frac12 \left( 
\frac{\bff_+^2}{\bu_+} + \frac{\bff_-^2}{\bu_-} \right) - \frac{\bff^2}{\bu}\,.
\]
By Lemma \ref{l:convB-03} 
\[
\frac{\wt\cB_1(\bff_+, N_+) + \wt\cB_1(\bff_-, N_-)}{2} - \wt\cB_1(\bff, N)\ge 
\frac12 \cdot \frac{(\sd\bu)^2 \bff^2}{\bu^2\alpha(\bu^*/\bu)\bu^*}
. 
\]
Recalling that $\wt\cB =2 \wt\cB_1 + \wt\cB_2$ we then can see that \eqref{mainDyPrelim-01} gives exactly the conclusion of the lemma. 
\end{proof}

\subsection{The main difference inequality for the Bellman function}

%
%

\subsubsection{General form}

Let $\f$ and $\wt\cB$ be as above. 
\begin{lm}
\label{l:BellMainIneq-01}
Let $\bff, \bff_k \in\R$, $\gamma_k\in \R_+$ and the distribution functions $N$, $N_k$, $k=1, 2, \ldots, n$ satisfy
\[
\bff = \sum_{k=1}^n  \gamma_k \bff_k, \qquad N = \sum_{k=1}^n  \gamma_k N_k,  \qquad 
\sum_{k=1}^n \gamma_k =1.  \ 
\]
Then the function $\wt\cB$ introduced above satisfies
\begin{align}
\label{eq:maindiff_02}
-\wt\cB(\bff, N) + \sum_{k=1}^n \gamma_k \wt\cB(\bff, N_k) \ge \frac{1}{4}\cdot \frac{1}{\alpha(\bu^*/\bu)\bu^*} \left( \sum_{k=1}^n \gamma_k | \bff_k -\bff |  \right)^2 
\end{align}
\end{lm} 

This lemma follows from the following general fact about convex function and is not specific to the function $\wt\cB$ introduces above. 

\begin{lm}
\label{l:2-to-n}
Let $\cN$ be an affine space, and let $D$ be a convex subset of $\R\times \cN$. 

Let $\cB$ be a  function on $D$, continuous on any finite-dimensional affine submanifold of $D$, and such that for all $(\bff, N), (\bff_\pm, N_\pm)\in D \subset \R\times \cN$ satisfying 
\[
\bff= (\bff_+ +\bff_-)/2, \qquad N= (N_+ + N_-)/2
\]
we have 
\begin{align}
\label{dy-assume}
\frac{\cB(\bff_+, N_+) + \cB(\bff_-, N_-) }{2} - \cB(\bff, N) \ge c(\bff, N) |\bff_+ - \bff|^2 \ge 0
\end{align}

Then for all $\bff, \bff_k \in\R$, $\gamma_k\in \R_+$ and the distribution functions $N$, $N_k$, $k\in\N$ satisfying
\[
\bff = \sum_{k=1}^n  \gamma_k \bff_k, \qquad N = \sum_{k=1}^n  \gamma_k N_k,  \qquad 
\sum_{k=1}^n \gamma_k =1.  \ 
\]
The following estimate holds:
\begin{align}
\label{MainIneq-abstract}
-\cB(\bff, N) + \sum_{k=1}^n \gamma_k \cB(\bff_k, N_k) \ge \frac{1}{4}\cdot c(\bff, N) \left( \sum_{k=1}^n \gamma_k | \bff_k -\bff |  \right)^2 
\end{align}
\end{lm}

\begin{proof}
The reasoning below is a ``baby version'' of the reasoning used to prove the main estimate (Lemma 6.1) in \cite{TB2011}.   

For a weight  $\gamma = \{\gamma_k\}_{k=1}^n$, $\gamma_k\ge0$, let $\ell^p(\gamma)$ be the (finite-dimensional) weighted  $\ell^p$ spaces, 
\[
\|x\|_{\ell^p(\gamma)}^p =\sum_{k=1}^n \gamma_k |x_k|^p, \qquad \|x\|_{\ell^\infty(\gamma)} = \sup\{|x_k|  : \gamma_k\ne 0\}
\]
(of course we need to take the quotient spaces over $\{x:\|x\|_{\ell^p(\gamma)}=0\}$). 

Let $\La\fdot, \fdot \Ra_\gamma$ be the standard duality $\La x, y\Ra_\gamma = \sum_{k=1}^n \gamma_k x_k y_k$. 

Define $\be\in \ell^p(\gamma) $, $\be = (1, 1, \ldots, 1)$. 

Consider the quotient space $\cX = \ell^1(\gamma)/\spn\{\be\}$. For $x\in \ell^1(\gamma)$ let 
\[
x^0:= x- \|\be\|_{\ell^1(\gamma)}^{-1} \La x, \be \Ra_\gamma \be, 
\] 
so $\sum_{k=1}^n \gamma_k x^0_k =0$. Then
\begin{align}
\label{eq-QuotNorm}
\|x\|\ci{\cX} \le \|x^0\|_{\ell^1(\gamma)} \le 2\|x\|\ci{\cX}.	
\end{align}
Indeed, the first inequality is trivial (follows from the definition of the norm in the quotient space). As for the second one, $|\La x, \be\Ra_\gamma| \le  \|x\|_{\ell^1(\gamma)}$, 
so it follows from the triangle inequality that
\[
\|x^0\|_{\ell^1(\gamma)} \le \|x\|_{\ell^1(\gamma)} + \|\be\|_{\ell^1(\gamma)}^{-1} |\La x, \be \Ra | \cdot \|\be\|_{\ell^1(\gamma)} \le 2 \|x\|_{\ell^1(\gamma)}. 
\] 
This inequality remains true if one replaces $x$ by $x-\beta\be$, $\beta\in\R$, so the second inequality in \eqref{eq-QuotNorm} is proved. 

The dual space $\cX^*$ can be identified with s subspace of $\ell^\infty(\gamma)$ consisting of $x^*\in \ell^\infty(\gamma)$ such that $\La \be, x^*\Ra_\gamma =\sum_k \gamma_k x_k =0$ (with the usual $\ell^\infty(\gamma)$-norm).

So, for the vector $x= \{x_k\}_{k=1}^n$, $x_k =\bff_k -\bff$ (notice that $\La x, \be\Ra_\gamma =0$) there is $\beta =\{\beta_k\}_{k=1}^n$, $|\beta_k |\le 1$ such that $\sum_{k=1}^n \gamma_k \beta_k =0$ and 
\[
\sum_{k=1}^n \gamma_k \beta_k (\bff_k-\bff) = \|x\|\ci{\cX} \ge \frac12 \| x\|_{\ell^1(\gamma)} =\frac12  \sum_{k=1}^n \gamma_k |\bff_k -\bff|. 
\]

Define $\bff_+$, $\bff_-$, $N_+$, $N_-$ by 
\begin{align}
\label{ConvComb-01}
\bff_\pm = \sum_{k=1}^n \gamma_k (1 \pm \beta_k) \bff_k, \qquad N_\pm := \sum_{k=1}^n \gamma_k (1 \pm \beta_k) N_k  
\end{align}
(note that $\sum_k \gamma_k (1\pm \beta_k) =1$ so $(\bff, N)$ is in the convex hull of $(\bff_k, N_k)$).

By the assumption  \eqref{dy-assume} of the lemma 
\begin{align}
\label{MainIneq_01}
\frac12 \Bigl( \cB(\bff_+, N_+))   +  \cB(\bff_-, N_-)) \Bigr) - \cB(\bff, N) \ge c(\bff, N)   (\bff_+ -\bff)^2. 
\end{align}
We know that 
\begin{align*}
| \bff_+ -\bff | = \sum_{k=1}^n \gamma_k \beta_k \bff_k = \sum_{k=1}^n \gamma_k \beta_k (\bff_k -\bff) \ge \frac 12 
\sum_{k=1}^n \gamma_k |\bff_k -\bff |
\end{align*}
(the second equality holds because $\sum_{k=1}^n \gamma_k\beta_k=0$), so the right side of \eqref{MainIneq_01} is estimated below by
\[
\frac{1}{4}  c(\bff, N) \left( \sum_{k=1}^n \gamma_k | \bff_k -\bff |  \right)^2
\]
The assumption \eqref{dy-assume} together with continuity on lines implies that  the function $\cB$ is convex, so we can conclude from \eqref{ConvComb-01} that 
\begin{align*}
\cB(\bff_+, N_+) &\le  \sum_{k=1}^n \gamma_k (1 + \beta_k) \cB(\bff_k, N_k),  \\ 
\cB(\bff_-, N_-) &\le  \sum_{k=1}^n \gamma_k (1 - \beta_k) \cB(\bff_k, N_k)
\end{align*}
and adding these inequalities we can estimate above the left side of \eqref{MainIneq_01} by 
\[
-\cB(\bff, N) + \sum_{k=1}^n \gamma_k \cB(\bff_k, N_k)  . 
\]
\end{proof}

\subsection{From main inequality to the embedding theorem}

The embedding theorem (Theorem \ref{t:H1_bumps-01}) follows from Lemma \ref{l:BellMainIneq-01} by the standard Bellman function reasoning.

We first prove the theorem under the assumption that each $I\in\cD$ has finitely many children. 
Let $I^0\in\cD$ and let $I_k$ be its children. As above, assume the normalization $C_\alpha=1$ for $\alpha$. 

Applying Lemma \ref{l:BellMainIneq-01} to $\bff=\La fu\Ra\ci{I^0}$, $\bff_k=\La bu\Ra\ci{I_k}$, $N=N_{I^0}^u$, $N_k=N_{I_k}^u$ (recall that $N_I^u $ is the normalized distribution function, see\eqref{NormDistrFn}) with $\gamma_k =|I_k|/|I^0|$ we get after multiplying by $|I^0|$
\[
\frac{1}{4} \cdot \frac{\|\Delta\ci{I^0} (fu)\|\ci{L^1(I^0)}^2 }{\alpha(\bu^*\ci{I^0}/\bu\ci{I^0}) \bu^*\ci{I^0}  } \, |I^0| \le
\sum_{I\in\ch(I^0)} |I| \wt\cB(\La fu\Ra\ci I, N_I^u ) \, |\  - \ |I^0| \, \wt\cB(\La fu\Ra\ci{I^0}, N_{I^0}^u )
\]
Applying this formula to all children of $I^0$, then to their children and using the telescoping sum in the right side we get after going $n$ generations down  that 
\begin{align*}
\frac{1}{4} \sum_{\substack{I\in \ch_k(I^0) \\ 0\le k < n }} 
 \frac{\|\Delta\ci{I} ( fu) \|\ci{L^1(I)}^2}{\alpha(\bu^*\ci{I}/\bu\ci{I}) \bu^*\ci{I}  } \, |I| 
 &\le 
\sum_{I\in\ch^n(I^0)} |I| \wt\cB(\La fu\Ra\ci I, N_I^u ) \, |\  - \ |I^0| \, \wt\cB(\La fu\Ra\ci{I^0}, N_{I^0}^u )
 \\
 & \le 
 \sum_{I\in\ch_n(I^0)} |I| \cdot \wt\cB(\La  f u \Ra\ci I, N_I^w ) \le 9 \sum_{I\in\ch_n(I^0)} 
|I| \frac{|\La fu\Ra\ci I|^2}{\La u\Ra\ci I};
\end{align*}
in the last inequality we used the property \eqref{bound_B} of $\wt\cB$. 

It follows from the Cauchy--Schwartz that 
\[
\frac{|\La fu\Ra\ci I|^2}{\La u\Ra\ci I} \le \La F^2u\Ra\ci I = |I|^{-1}\int_I F^2 u , 
\]
so the previous inequality means that 
\[
\sum_{\substack{I\in \ch_k(I^0) \\ 0\le k < n }} 
 \frac{\|\Delta\ci{I} ( fu) \|\ci{L^1(I)}^2}{\alpha(\bu^*\ci{I}/\bu\ci{I}) \bu^*\ci{I}  } \, |I| \le 36 \int_{I^0} f^2 u dx = 36 \|f\1\ci{I^0}\|\ci{L^2(u)}. 
\]
Then, letting first $n\to\infty$ and then taking the sum over all $I^0\in \cD_{-m}$ and letting $m\to \infty$ we get the conclusion of the theorem. 

To prove the result in general situation, when the intervals can have infinitely many children, we  prove the result for the functions  $f$ which can be represented as  finite sums
\[
f=\sum_{I\in \cD_r} c\ci I \1\ci I  
\]
($n$ is not fixed). In this case any interval $J\in\cD_n$, $n<r$ has only finitely many children containing $I$s from the above sum. On the other hand, nothing will change in the left hand side of \eqref{embed-02} if we treat the rest of the children of $J$ as one interval, so we arrive to the case when each interval has finitely many children, which we already proved. 

Standard approximation reasoning gives us that for all $r$ the theorem holds for functions constant on the intervals $I\in\cD_r$. Since for 
\[
\E_r^u f := \sum_{I\in\cD_r} \frac{\La fu\Ra\ci I}{\La u\Ra\ci I} \1\ci I 
\] 
we have $\|\E^u_r f \|\ci{L^2(u)} \le \|f\|\ci{L^2(u)}$, applying the theorem for $\E^u_r f$ we get the estimate for the sum over all $I\in \cD_n$, $n<r$. Letting $r\to\infty $ give the theorem. \hfill\qed

\subsection{Two sided bumps for paraproducts}

Theorem \ref{t:main-01} for paraproducts is immediately obtained by combining Embedding Theorems \ref{t:embed-03} and \ref{t:H1_bumps-01} via Cauchy--Schwarz. 

Indeed, denote 
\[
\bu^*\ci I := \| u\|\ci{\Lambda_{\psi_0}(I)} \ge \|M\1\ci I\|_1, \qquad \psi_0(s) = s\ln(e/s), 
\]
and similarly for $v$. 
For $f\in L^2(u)$, $g\in L^2(v)$ we get assuming  \eqref{entr-bump-01} with $A=1$, 
\begin{align*}
\left|(\Pi_b fu, g)\ci{L^2(v)} \right| 
& \le \sum_{I\in\cD} |\La fu\Ra\ci I| \cdot \left|\left(b\ci I u, \Delta\ci I (bv) \right)\ci{L^2} \right| 
\\
& \le \sum_{I\in\cD} \frac{|\La fu\Ra\ci I|\cdot \|b\ci I\|_\infty |I|^{1/2} }{\bigl(\alpha(\bu^*\ci I/\bu\ci I) \bu^*\ci I \bigr)^{1/2}} \cdot 
\frac{ \|\Delta\ci I (gv) \|\ci{L^1(I)} |I|^{1/2} }{ \bigl(\alpha(\bv^*\ci I/\bv\ci I) \bv^*\ci I \bigr)^{1/2}}
\\
& \le 
\left(  \sum_{I\in\cD} \frac{|\La fu\Ra\ci I|^2 \|b\ci I\|_\infty^2 |I| }{\alpha(\bu^*\ci I/\bu\ci I) \bu^*\ci I} \right)^{1/2} 
\left( \sum_{I\in\cD} \frac{ \|\Delta\ci I (gv) \|\ci{L^1(I)}^2 |I| }{ \alpha(\bv^*\ci I/\bv\ci I) \bv^*\ci I} \right)^{1/2}; 
\end{align*}
the second inequality holds because of \eqref{entr-bump-01} with $A=1$ (note also that $\|gv\|_1= \|gv\|\ci{L^1(I) }|I|$), and the last one is just the Cauchy--Schwarz.

Applying to the sums in parentheses Theorems \ref{t:embed-03} and \ref{t:H1_bumps-01} respectively we get Theorem \ref{t:main-01} for paraproducts. 

Gathering together estimates we can get the constant $C$ in \eqref{inte-02} to be equal to $(24C_\alpha A^{1/2})^2$, where $C_\alpha$ is defined by \eqref{C_alpha} and $A$ is the supremum in \eqref{entr-bump-01}. \hfill \qed

\section{One-sided bumps for the Lerner type operators}

Let $\cQ$ be a sparse collection, and let $T=T\ci\cQ$ be the corresponding Lerner type (sparse) operator. 

As above in Section \ref{s:lerner-2side}, for a weight $u$ let $\bu\ci I:= \La u\Ra\ci I = \|u\|\ci{L^1(I)}$, 
\[
\bu^*\ci I:= 
\| M \1\ci I u \|\ci{L^1(I)} \le
\|u\|_{\Lambda_{\psi_0}(I)}, \qquad \psi_0(s) = s\ln(e/s), 
\] 
where $M$ is the martingale maximal function and similarly for a weight $v$. 

Let also as above $\alpha:[1, \infty)\to \R_+$ be a function such that $t\mapsto t\alpha(t)$ is increasing and 
\begin{align}
\label{C_alpha-mod}
C_\alpha:=\int_1^\infty \frac1{t\alpha(t)} dt <\infty. 
\end{align}

To prove Theorem \ref{1sided-02} we first recall that a sparse operator is a particular case of the so-called \emph{positive dyadic} operators, 
and for such operators \eqref{1sided-02} holds if and only if the so-called \emph{Sawyer type testing conditions}
\begin{align}
\label{SawTestCond}
\int_I |T(\1\ci I u)|^2 v dx & \le S \cdot \|\1\ci I\|^2\ci{L^2(u)} = S \cdot\La u \Ra\ci I |I|, \\
\int_I |T(\1\ci I v)|^2 u dx & \le S \cdot \|\1\ci I\|^2\ci{L^2(v)} = S \cdot\La v \Ra\ci I |I|
\end{align}
are satisfied  for some $S<\infty$ for all $I\in\cD$; moreover the constant $C$ in  \eqref{1sided-02} can be estimated by $KS$, where $K$ is an absolute constant  ($K= \left(8(2+\sqrt2)\right)^2$ can be obtained by tracking estimates in \cite{Tr_PosDy_2012}).

\begin{thm}
\label{t:TestLernOp_p=2}
Let $u$, $v$ be weights such that 
\begin{align}
\label{Bumped_H1_bump}
\sup_{I\in\cQ}  \alpha(\bu\ci I^*/\bu\ci I)^2 \bu^*\ci I  \bv\ci I = A< \infty
\end{align}
Then for any $I_0\in \cD$ 
\[
\int_{I_0} |T (\1\ci{I_0}u)|^2 v dx \le C C_\alpha^2 A\, \|1\ci{I_0}\|\ci{L^2(u)}^2 = C C_\alpha^2 A \, \bu\ci{I_0} |I_0|, 
\]
where $C_\alpha$ is given by \eqref{C_alpha-mod} and $C$ is an absolute constant. 
\end{thm}

%

\subsection{Some preliminaries}
Before proving the theorem we need to introduce some notation. Let us construct the family of stopping moments $\cG=\cG(I_0)$ as follows. 

For an interval (a cube) $J$ we denote $\cG^*(J)$ to be the collection of maximal (by inclusion) intervals $I$ such that 
\begin{align}
\label{StopMoment}
\La u \Ra\ci I \ge 2\La u\Ra\ci J. 
\end{align}
Then we define the generations of stopping moments $\cG_k$ inductively, $\cG_0 =\{I_0\}$, 
\begin{align}
\label{cG_k}
\cG_{k+1} = \bigcup_{I \in \cG_k}\cG^*(I),  
\end{align}
and put $\cG:= \bigcup_{k\ge0} \cG_k$.

For an interval $J\in \cQ$ denote 
\[
\cQ(J):= \{I\in \cQ:I\subset J\}
\]
(note that $J\in\cQ(J)$), and define
\begin{align}
\label{cE(J)}
\cE(J):= \cQ(J)\setminus \bigcup_{I\in\cG^*(J)}\cQ(I). 
\end{align}

For $I\in\cG$ define
\begin{align}
\label{U_J}
U\ci J = U\ci{\cE(J)} := \sum_{I\in\cE(J)} \La u\Ra\ci I \1\ci I. 
\end{align}

\begin{lm}
\label{l:carl-U_J}
Let
\begin{align}
\label{Muck-01}
\sup_{I\in\cQ(I_0)} \La u\Ra\ci I \La v\Ra\ci I = A_1<\infty. 
\end{align}
Then 
\[
\sum_{J\in\cG(I_0)} \|U\ci J\|\ci{L^2(v)}^2 \le C A_1 \bu^*\ci{I_0} |I_0|\,,
\]
where $C$ is an absolute constant. 

We should emphasize that we do not assume \eqref{Bumped_H1_bump} here. 
\end{lm}


To prove Lemma \ref{l:carl-U_J} we need the following simple fact: 

\begin{lm}
\label{l:norm-U_J}
Under the assumptions of Lemma \ref{l:carl-U_J}
\[
\|U\ci J\|\ci{L^2(v)}^2 \le C_1 A_1 \La u \Ra\ci J |J| .
\]
where $C_1$ is an absolute constant. 
\end{lm}

\begin{proof}
Let 
$\cQ_k(J)$ be the $k$th generation of the $\cQ$-(grand)children  of $J$: 
$\cQ_0(J)=\{J\}$, $\cQ_1(J)$ be the collection of maximal (by inclusion) $I\in \cQ$ such that $I\subsetneq J$, and
\[
\cQ_{k+1} (J) = \bigcup_{I\in\cQ_k(J)} \cQ_1(I) \,.
\]
Note, that since $\cQ$ is a sparse family
\begin{align}
\label{GeomProgr-01}
\sum_{I\in\cQ_k(J)} |I| \le 2^{-k} |J| 
\end{align}
Clearly, 
\[
U\ci J = \sum_{k\ge 0} \sum_{I\in\cQ_k(J)} \La u\Ra\ci I \1\ci I =: \sum_{k\ge0} U\ci{J,k}. 
\]
We can estimate using \eqref{Muck-01}
\begin{align*}
\|U\ci{J,k}\|\ci{L^2(v)}^2 & = \sum_{I\in\cQ_k(J)} \La u \Ra\ci I^2 \La v\Ra\ci I |I| && \\
& \le \sum_{I\in\cQ_k(J)} \La u \Ra\ci I A_1 |I|  && \text{by \eqref{Muck-01}} \\
& \le 2 A_1\La u\Ra\ci J \sum_{I\in\cQ_k(J)} |I| && \text{because } \La u\Ra\ci I \le 2\La u\Ra\ci J \\
& \le 2 A_1 \La u\Ra\ci J \, 2^{-k} |J| && \text{by \eqref{GeomProgr-01}}. 
\end{align*}
Therefore
\[
\|U\ci J\|\ci{L^2(v)} \le \sum_{k\ge0} \|U\ci{J,k}\|\ci{L^2(v)} \le 2^{1/2}A_1^{1/2} \La u\Ra\ci J^{1/2} \sum_{k\ge0} 2^{-k/2} =\frac{2^{1/2}}{1-2^{-1/2}} A_1^{1/2} \La u\Ra\ci J^{1/2}
\]
and the lemma is proved with $C_1=2/(1-2^{-1/2})^2$. 
\end{proof}

\begin{proof}[Proof of Lemma \ref{l:carl-U_J}]
Applying Lemma \ref{l:norm-U_J} we get
\[
\sum_{J\in\cG(I_0)} \|U\ci J\|\ci{L^2(v)}^2 \le C_1 A_1 \sum_{I\in\cQ(I_0)} \La u \Ra\ci I |I| . 
\]
Since $\cQ$ is sparse, the corresponding measure is Carleson, 
so
\[
\frac{1}{|I_0|}\sum_{I\in\cQ(I_0)} \La u \Ra\ci I |I| \le C \|M(\1\ci{I_0}u)\|\ci{L^1(I_0)} \le C\bu^*\ci{I_0} . 
\]
\end{proof}

\subsection{Proof of Theorem \ref{t:TestLernOp_p=2}: first splittings and the easy estimate}
\label{s:FirstSplit}

The idea of the proof is as follows. We split the operator $T$ as the sum $T=\sum_{k,n\ge 0} T_{k,n}$, where in each $T_{k,n}$ the summation is taken only over the intervals $I\in \cQ$ such that 
\begin{align}
\label{bds-rho}
2^k \le \rho\ci I := \bu^*\ci I /\bu_I &< 2^{k+1} \,,  
\\
\label{bds-uv}
2^{-n-1}B_k < \bu\ci I \bv\ci I  &\le  2^{-n} B_k  \,,
\end{align}
where $B_k = 2^{-k} \alpha(2^{k})^{-2} A $. Rewriting the assumption \eqref{Bumped_H1_bump} as
\begin{align}
\label{Bumped_H1_bump-01}
\sup_{I\in\cQ} \alpha(\rho\ci I)^2 \rho\ci I \bu\ci I \bv\ci I =A <\infty
\end{align}
we can see that the inequality $\rho\ci I \ge 2^k$ from \eqref{bds-rho}  implies that $\bu\ci I \bv\ci I \le B_k$, so indeed $T=\sum_{k,n\ge0} T_{k,n}$.


We will prove that under the assumption of Theorem \ref{t:TestLernOp_p=2} 
\begin{align}
\label{Norm-T_kn-u}
\| T_{k,n} (\1\ci{I_0}u) \|\ci{L^2(v)} \le C A^{1/2} 2^{-n/2}\alpha(2^k)^{-1} \bu_{I_0}^{1/2}|I_0|^{1/2}.
\end{align}
The theorem will immediately follows from the following simple fact

\begin{lm}
\label{sum-alpha_2^k}
Let $t\mapsto t\alpha(t) > 0$ be an increasing function on $[1, \infty)$. Then  
\[
 \sum_{k=1}^\infty \frac{1}{\alpha(2^k)} \le 2 \int_1^\infty \frac{dt}{t\alpha(t)}  . 
\]
\end{lm}
\begin{proof}
Split $[1,\infty)$ into the intervals $[2^k, 2^{k+1})$, $k\ge 0$ consider the lower (right) Riemann sum in the integral. 
\end{proof}

So, let us skip the indices $k$ and $n$  and assume from now that \eqref{bds-rho} and \eqref{bds-uv} hold for all $I\in \cQ$. The stopping moments $\cG$ and the functions $U\ci J$ are defined as before. Applying 
Lemma \ref{l:carl-U_J} and using the second inequality in  \eqref{bds-rho} we get that
\[
\sum_{J\in\cG(I_0)} \|U\ci J\|\ci{L^2(v)}^2 \le C A_1 \bu^*\ci{I_0} |I_0|\le 2^{k+1}  C A_1 \bu\ci{I_0} |I_0|\,.
\]
The second inequality in \eqref{bds-uv} implies that 
\[
A_1 \le 2^{-n} B_k = 2^{-n} 2^{-k} \alpha(2^k)^{-2} A
\]
so
\begin{align}
\label{carl-U_J-01}
\sum_{J\in\cG(I_0)} \|U\ci J\|\ci{L^2(v)}^2 \le 2 C A 2^{-n} \alpha(2^k)^{-2} \bu\ci{I_0} |I_0|\,.
\end{align}

For $J\in \cG$ define $G(J):= \bigcup_{I\in \cG^*(J)} I$. Denoting $d\nu = vdx$ take $g\in L^2(v)$, $\|g\|\ci{L^2(v)} =1$ and write
\begin{align*}
\int U\ci J g d\nu = \int_J U\ci J g d\nu = \int_{J\setminus G(J)} U\ci J g d\nu +
\int_{G(J)} U\ci J g d\nu =: A(J) + B(J). 
\end{align*}

To estimate the sum of $A(J)$ let us write  
\begin{align*}
\sum_{J\in\cG} A(J)  & \le  \sum_{J\in\cG(I_0)} \| U\ci{J} \|\ci{L^2(\nu)} \| g \1\ci{J\setminus G(J)} \|\ci{L^{2}(\nu)} 
\\ & \le
\left(  \sum_{J\in\cG(I_0)} \| U\ci{J} \|\ci{L^2(\nu)}^2 \right)^{1/2} 
\left(  \sum_{J\in\cG(I_0)}   \| g \1\ci{J\setminus G(J)} \|\ci{L^{2}(\nu)}^{2}  \right)^{1/2}  && \text{Cauchy--Schwarz}
\\  &  \le 
\left(  \sum_{J\in\cG(I_0)} \| U\ci{J} \|\ci{L^2(\nu)}^2 \right)^{1/2} \| g  \|\ci{L^{2}(\nu)} && 
J\setminus G(J)\ \text{are disjoint}
\\  & \le
C A^{1/2} 2^{-n/2}\alpha(2^k)^{-1} \bu_{I_0}^{1/2}|I_0|^{1/2} 
&& \text{by \eqref{carl-U_J-01} } , 
\end{align*}
so we got the correct estimate for the sum of $A(J)$. 

\subsection{Conclusion of the proof of Theorem \ref{t:TestLernOp_p=2}: 
the ``hard'' estimate}
\label{s:HardEst}

The estimate  the sum of $B(J)$ 
is based on the fact that the system $\cG$ of stopping intervals (cubes) is $\nu$-Carleson, meaning that for any $J\in \cQ$
\begin{align}
\label{Carl-v-01}
\sum_{I\in \cG,\, I\subset J} \nu(I) \le C \nu(J) . 
\end{align}
To see that we notice that for $I\in \cG^*(J)$ we have 
\[
\bu\ci I \ge 2 \bu\ci J . 
\]
On the other hand we can see from \eqref{bds-uv} that
\[
\bu\ci I \bv\ci I \le 2^{-n} B_k, \qquad \bu\ci J \bv\ci J \ge 2^{-n-1}B_k, 
\]
so 
\[
\bv\ci I /\bv\ci J \le 2 \bu\ci J/\bu\ci I \le 1. 
\]
Therefore for $J\in \cG$
\begin{align*}
\sum_{I\in \cG^*(J)} \nu(I) = \sum_{I\in\cG^*(J)} \bv\ci I |I| \le \bv\ci J \sum_{I\in\cG^*(J)}  |I| \le \bv\ci J  |J|/2  =\nu(J)/2 .
\end{align*}
Summing the geometric series we get that for any $J\in\cG$  estimate \eqref{Carl-v-01} holds with $C=2$. To get the estimate for arbitrary $J\in\cQ$, we apply the estimate we just proved to the maximal (by inclusion) $\wt J\in \cG$, $\wt J\subset J$. 

To complete the proof let us denote
\[
\La k \Ra\ci{I,\nu} := \nu(I)^{-1} \int_I g d\nu  = \La gv\Ra\ci I/ \La v\Ra\ci I .
\]
Since the function $U\ci J$ is constant on any stopping interval $I\in \cG^*(J)$, 
\[
B(J) = \int_{G(J)} U\ci J g d\nu = \int_{G(J)} U\ci J g\ci J d\nu, 
\]
where 
\[
g\ci J := \sum_{I\in\cG^*(J)} \La g \Ra\ci{I,\nu} \1\ci I .
\]
Note that 
\[
\| g\ci J\|\ci{L^2(\nu)}^2 = \sum_{I\in\cG^*(J)} \La g \Ra\ci{I,\nu}^2 \nu(I). 
\]

Then we can estimate
\begin{align*}
\sum_{J\in \cG(I_0)} B(J) & \le \sum_{J\in \cG(I_0)} \|U\ci J\|\ci{L^2(\nu)} \|g\ci J\|\ci{L^2(\nu)}  \\
& \le \left( \sum_{J\in \cG(I_0)} \|U\ci J\|\ci{L^2(\nu)}^2 \right)^{1/2} 
 \left( \sum_{J\in \cG(I_0)} \|g\ci J\|\ci{L^2(\nu)}^2 \right)^{1/2}
\end{align*}
The first factor is already estimated in \eqref{carl-U_J-01}. To estimate the second factor
we write 
\begin{align*}
\sum_{J\in \cG(I_0)} \|g\ci J\|\ci{L^2(\nu)}^2 & = \sum_{J\in \cG(I_0)} \sum_{I\in \cG^*(J) } \La g \Ra\ci{I,\nu}^2 \nu(I) \\
& = \sum_{I\in \cG(I_0), I\ne I_0} \La g \Ra\ci{I,\nu}^2 \nu(I) \le C \|g\|\ci{L^2(\nu)}^2 = C; 
\end{align*}
the inequality follows from the martingale Carleson embedding theorem. Note that one can take $C=8$ in this estimate: $2$ is a constant in \eqref{Carl-v-01}, and $4$ is the constant in the embedding theorem. 

Combining this with the estimate \eqref{carl-U_J-01} for the first factor we get
\[
\sum_{J\in\cG(I_0)} B(J) \le C A^{1/2} 2^{-n/2} \alpha(2^k)^{-1} \bu\ci{I_0}^{1/2} |I_0|^{1/2} 
\]
(square root of the estimate \eqref{carl-U_J-01}). Theorem \ref{t:TestLernOp_p=2} is proved. \hfill\qed

\section{Remarks about one weight estimates}
\label{s:OneWeight}

Theorem \ref{t:main-01} implies that one weight dyadic Muckenhoupt condition
\begin{align}
\label{A_2}
\sup_{I\in\cD} \La v\Ra\ci I \La v^{-1} \Ra\ci I =: [v]\ci{A_2} <\infty
\end{align}
implies the boundedness of the Haar shifts, paraproducts and sparse operators 
in the weighted space $L^2(v)$, or, equivalently, the boundedness of the 
operator $M_v^{1/2} T M_v^{-1/2}$, equivalently  the estimate \eqref{inte-02} 
with $u=v^{-1}$. 

In the homogeneous case this result is well known, but it is new in the 
non-homogeneous situation. Note also that in the definition of Haar shifts we 
use $L^1\times L^1 $ normalization  \eqref{L1xL1-norm} of the blocks $T\ci I$; 
the question whether the condition \eqref{A_2} is sufficient if we only assume 
that  the blocks $T\ci I$ are uniformly bounded in non-weighted $L^2$ remains 
open. At the moment we do not even know whether \eqref{A_2} is sufficient for 
the uniform boundedness of the martingale multipliers ($T\ci I=\pm\bI$). 
All this shows that the one weight non-homogeneous case warrants further 
investigation.


To prove the result mentioned at the beginning of this section, one can consider the Wilson's $A_\infty$ characteristic of a weight $v$
\begin{align}
\label{W-A_infty}
[v]\ci{A_\infty} := \sup_{I\in\cD} \frac{\|M(\1\ci I v)\|\ci{L^1(I)} }{\La v \Ra\ci I} \le [v]\ci{A_2};
\end{align}
the last inequality is well-known and is not hard to prove. Since $[v^{-1}]\ci{A_2} = [v]\ci{A_2}$, 
we conclude that $[v^{-1}]\ci{A_\infty} \le [v]\ci{A_2}$. 

Thus, defining for example $t\alpha(t):=[v]\ci{A_2}$ for $t\le [v]\ci{A_2}$ and 
$\alpha(t)=\infty$ for $t>[v]\ci{A_2}$ and using Theorems \ref{t:H1_bumps-01} 
and we get the boundedness with the estimate $C[v]\ci{A_2}^{3/2}$ of the norm. 

We were not able to get the linear estimate $C[v]\ci{A_2}$ of the norm in the 
non-homogeneous situation, which probably should be the correct one, by picking 
an appropriate $\alpha$. However, for the sparse operators the linear estimate 
of the norm can be obtained by an obvious modification (and simplification) of 
the proof of Theorem \ref{t:TestLernOp_p=2}. And this does not depend on whether we are in homogeneous or non-homogeneous setting.

\begin{rem*} However, we remind the reader that in the non-homogeneous 
situation one does not know whether the treatment of the weighted boundedness 
of any Calder\'on--Zygmund operator can be reduced to the treatment of sparse 
operators of Lerner's type. As a contrast, there is a reduction of a general 
non-homogeneous Calder\'on--Zygmund operator to martingale shifts and 
paraproducts, see \cite{Vo1}.
\end{rem*}

Coming back to the sparse operators, if we assume that 
\[
\sup_{I\in\cD}\bu\ci I \bv\ci I =:[u,v]\ci{A_2} <\infty , 
\]
and that 
\[
\sup_{I\in\cD} \bu^*\ci I/\bu\ci I =: [u]\ci{A_\infty}<\infty,  
\]
then we can show that for a sparse operator $T$ for any $I_0\in\cD$ 
\[
\int_{I_0} |T (\1\ci{I_0} u) |^2 v dx \le C [u]\ci{A_\infty} [u,v]\ci{A_2} \bu\ci{I_0} |I_0|.  
\]

Again, this result is well-known in the homogeneous situation, even for two weights, see for example

To show the above estimate, one  splits the operator $T=\sum_{n:\,  2^n \le[u]\ci{A_2}} T_n$, where the sum in $T_n$ is taken over all $I$ such that 
\[
 2^n \le \La u \Ra\ci I \La v \Ra\ci I \le 2^{n+1} . 
\]
Then is is sufficient to show that 
\[
\|T_n (\1\ci{I_0}u) \|\ci{L^2(v)} \le C [u]\ci{A_\infty}^{1/2} 2^{n/2}, 
\]
which can be done following the reasoning in Sections \ref{s:FirstSplit} and \ref{s:HardEst}. 

We leave the details as an exercise for the reader.

\section{Remarks about sharpness}
\label{s:Sharpness}

As we mentioned before, the $L\log L$ bump (even two-sided) is not sufficient 
for the boundedness  of \cz operators. Moreover, we will show that for any 
rearrangement invariant Banach  spaces $X$ and $Y$  on a unit interval where the fundamental functions $\psi$ of $X$ is such that 
\[
\int_0 \frac{ds}{\psi(s)} =+\infty
\] 
there exists a pair of weights $u$, $v$ on $\R$ such that for all intervals $I$
\begin{align}
\label{bumps_XY}
\|u\|\ci{X(I)} \|v\|\ci{Y(I)} \le B<\infty  
\end{align}
but 
\begin{align}
\label{T1=infty}
\|T(1\ci{[-1,1]}u ) \|\ci{L^2(v)} = \infty ;
\end{align}
here $T$ is the Hilbert Transform, 
$X(I) = X(I, \frac{dx}{|I|})$ and similarly for $Y(I)$. 

Thus the operator $f\mapsto T(uf)$ does not act $L^2(u)\to L^2 (v)$. Moreover, even a weak 
type estimate (for the adjoint operator) fails. Recall that the adjoint operator 
(acting ($L^2(v)\to L^2(u)$) is given by $g\mapsto -T(gv)$: and 
 \eqref{T1=infty} implies that the operator $g \mapsto T(gv)$ does not act 
$L^2(v) \to L^{2,\infty}(u)$ (i.e.~it is not even of weak type $2$-$2$). 

This follows from a  well-known reasoning: if the operator acts $L^2(v) \to L^{2,\infty}(u)$, 
then computing integral using distribution function and using weak type estimates, 
one can conclude that for any measurable $E$
\[
\int_E |T(gv)|^2 u \le C \|g\|\ci{L^2(v)}^2 \qquad \forall f\in L^2(v). 
\]
But this exactly means that 
\[
\| |T(\1\ci E u)\|^2v \le C\|\1\ci E\|\ci{L^2(u)}^2 = C\int_E u 
\]
for all measurable $E$. But in our example below this fails even for the interval $[-1,1]$.

To construct the example  define $u:=\1\ci{[-1,1]}$. It follows from the definition 
of the fundamental function $\psi=\psi\ci X$ that for $I=[0,a]$ or $I=[-a,a]$, $a\ge 1$
\begin{align}
\label{X-norm_u_on_I} 
\| u\|\ci{X(I)} = \psi(1/a)
\end{align}

Defining 
\[
v(x):=\left\{ \begin{array}{ll} 1/\psi(1/|x|) , \qquad & |x|\ge 1, \\
                                1/\psi(1) ,   & |x|<1 ,    \end{array}\right.
\]
we can see that \eqref{bumps_XY} is satisfied. 

Indeed, if $I=[0,a]$ or $I=[-a,a]$ then the estimate \eqref{bumps_XY} 
follows immediately from  \eqref{X-norm_u_on_I} , definition of $v$ and the 
property of Banach function spaces $\|v\|\ci{Y(I)} \le C\|v\|\ci{L^\infty(I)}$.  

If $I\cap[-1,1] =\varnothing$ then \eqref{bumps_XY} is trivial (left hand 
side is zero), so one needs to show that \eqref{bumps_XY} holds uniformly 
for all $I$, $I\cap[-1,1] \ne\varnothing$.  

If $|I|> 2$ (and $I\cap[-1,1] \ne\varnothing$) then, denoting $I_0=[-1,1]$
\begin{align*}
\| u \|\ci{X(I)} = \psi (|I\cap I_0|/ |I|) &\le \psi (|I_0|/|I|) && \text{because }
\psi(s)\uparrow
\\
&\le  4\psi (|I_0|/ (4|I|)) && \text{because } s\psi(s) \downarrow \\
& =  4\psi (1/ (2|I|)) \,\, , 
\end{align*}
and
\[
\|v\|\ci{Y(I)}\le C \|v\|\ci{L^\infty(I)} \le C \psi(1/(|I|+1))^{-1} \le C \psi(1/(2|I|))^{-1} .
\]
Combining these two inequalities we get \eqref{bumps_XY} with $B=4C$. 


If $|I|\le 2$ (and still $I\cap[-1,1] \ne\varnothing$) then 
\[
v(x) \le 1/\psi(1/3), \qquad u(x) \le 1  \qquad \forall x\in I, 
\]
so \eqref{bumps_XY} trivially holds for such intervals. 

Thus, \eqref{bumps_XY} holds for the weights $u$, $v$. 

Let now $I=[-1,1]$. Then $\1\ci I u=u$ and $\|\1\ci Iu\|^2\ci{L^2(u)}  = 2$. On 
the other hand for $|x|>1$
\[
|T\1\ci I u (x)|\ge 1/|x|, 
\]
so
\[
\| T\1\ci I u\|^2\ci{L^2(v)} \ge 2\int_1^\infty \frac{1}{x^2}\cdot\frac{1}{\psi(1/x) }dx = 2\int_0^1 
\frac{ds}{\psi(s)} =\infty, 
\]
and \eqref{T1=infty} is proved. \hfill \qed

\begin{rem}
A similar construction shows the necessity of the integrability condition 
\begin{align}
\label{inte-03}
\int^\infty \frac{dt}{t\alpha(t)} <\infty
\end{align}
in Theorem \ref{t:main-01}. 

Namely, suppose this condition fails for a penalty function $\alpha$, and $\beta$ is an arbitrary penalty function (that could satisfy \eqref{inte-03}). As usual we assume that $t\mapsto t\alpha(t)$ and $t\mapsto t\beta(t)$ are increasing; let us also assume that $t\mapsto e^t/(t\alpha(t))$ increases for $t\ge1$. 

Then it is possible to construct 
a pair of weights $u$, $v$ satisfying the bump condition 
\begin{align}
\label{bump-03}
\sup_I \alpha(\bu^*\ci I/\bu\ci I) \bu^*\ci I \bv^*\ci I \beta(\bv^*\ci I /\bv\ci I) <\infty
\end{align}
and such that 
\eqref{T1=infty} holds. 

To do that we again put $u=\1\ci{[-1,1]}$. Recalling that $\bu\ci I^* = \| u\|\ci{\Lambda_{\psi_0}(I)}$, 
$\psi_0(s) = s \ln(e/s)$, we compute using the definition of the fundamental function that for $I=[0,a]$, $a\ge 1$
\[
\bu\ci I^* = \|u\|\ci{\Lambda_{\psi_0(I)}} = a^{-1} \ln(e a), \qquad \bu\ci I=\|u\|\ci{L^1(I)} = a^{-1}.  
\]
Defining for $|x|\ge 1$
\[
v(x) = \frac{|x|}{\ln(e|x|) \alpha(\ln (e|x|))}
\]
and putting $v(x) = v(1)$ for $|x|<1$ we get the weights satisfying \eqref{bump-03} and such that \eqref{T1=infty} holds. 

The fact that \eqref{T1=infty} hods follows elementary from the failure of the integrability condition \eqref{inte-03}. The fact that the bump condition \eqref{bump-03} is satisfied can be proved similarly to how it was done in the previous example. 

The crucial fact there is that the weight $v$ is clearly doubling, $v(x)\le 2 v(2x)$, and it is also increasing by the assumption that $t\mapsto e^t/(t\alpha(t))$ increases for $t\ge 1$. Then for any interval $I=[0,a] $ 
\[
\|v\|\ci{L^1(I)} \asymp\|v\|\ci{\Lambda_{\psi_0}(I)} \asymp \|v\|\ci{L^\infty(I)} = \psi(a),  
\]
and so  it is easy to check that \eqref{bump-03} holds for intervals $I$ of form $[0,a]$ and $[-a,a]$. 

The general case can be reduced to this case. We need only to check that the intervals $I$ such that $I\cap [-1,1]\ne\varnothing$. Let us fix $|I|$ and consider the worst possible cases, i.e.~the maximal possible values of $\bu^*\ci I$ and $\bv^*\ci I$, and the minimal possible values of $\bu$ and $\bv$ (assuming that $|I|$ is fixed and $I\cap [-1,1]\ne \varnothing$). 

Then using doubling property of $v$ we get the conlclusion. 

We leave the detail as an exercise for the reader. 
\end{rem}

\begin{rem}
For general $p$, $1<p<\infty$, the bumping condition reads
\begin{align}
\label{pBump_XY}
\| u\|\ci{X(I) }^{1/p'} \|v\|\ci{Y(I)}^{1/p} \le B <\infty
\end{align}
uniformly on all intervals $I$; here $1/p+1/p'=1$. 

One can see that  exponents are correct by investigating the homogeneity. Note also that the case $X=Y=L^1$ gives the Muckenhoupt $A_p$ condition. 

If the fundamental function $\psi =\psi\ci X$ is such that 
\begin{align}
\label{pBump-psi}
\int_0^1 \psi(s)^{-p/p'} s^{p-2} ds = \infty
\end{align}
(nothing is assumed about $X$) one can construct a pair of weights $u$, $v$ satisfying \eqref{pBump_XY} and such that for the Hilbert Transform $T$ 
\begin{align}
\label{pT1=infty}
\| T(\1\ci{[-1,1]} u ) \|\ci{L^p(v)} = \infty. 
\end{align}

To do that we define $u:=\1\ci{[-1,1]}$ and $v(x) = \max\{ 1/\psi(1), 1/\psi(|x|)\}^{-p/p'}$. All the calculations are similar to the presented above for the $L^2$ case; we leave them as an exercise for the reader.  
\end{rem}

\begin{rem*}
In the above examples \eqref{T1=infty} and \eqref{pT1=infty} do hold if we replace the Hilbert Transform $T$ by the maximal function.
\end{rem*}

An example that for $\psi =\psi\ci{X}$ satisfying \eqref{pBump-psi} and for $Y=L^1$ the bumping condition \eqref{pBump_XY} is not sufficient for the weighted estimate for the maximal function was presented in \cite{PeMax}, see Proposition 3.2 there. 

This proposition is stated in a slightly different language, but  after translation one can see that   integrability of the left hand side in \eqref{pBump-psi} 
is equivalent to the integrability condition on $\f\ci{X'}$ in there; note that our $\psi$ and $\f\ci{X'}$ in \cite{PeMax} are related as
$\f\ci{X'}(s) = s\psi^{-1/p'}(s) $.

As for the Hilbert Transform the only counterexample  (previous to ours) we are aware of, is the example in \cite{CU-Pe99} showing that the condition \eqref{bumps_XY} with $X=L\log L$, $Y=L^1$ is not sufficient for 
the operator $g\mapsto T(fv)$ to be acting  $L^2(v)\to L^{2,\infty}(u)$.

\def\cprime{$'$}
  \def\lfhook#1{\setbox0=\hbox{#1}{\ooalign{\hidewidth\lower1.5ex\hbox{'}\hidewidth\crcr\unhbox0}}}
\providecommand{\bysame}{\leavevmode\hbox to3em{\hrulefill}\thinspace}
\providecommand{\MR}{\relax\ifhmode\unskip\space\fi MR }
\providecommand{\MRhref}[2]{%
  \href{http://www.ams.org/mathscinet-getitem?mr=#1}{#2}
}
\providecommand{\href}[2]{#2}

\end{document}